\documentclass[10pt]{amsart}
\usepackage{amsmath,amsfonts,amsthm,amssymb,amsxtra,hyperref,graphicx}
\newtheorem{thm}{Theorem}
\newtheorem{cor}[thm]{Corollary}
\newtheorem{lem}[thm]{Lemma}
\newtheorem{prop}[thm]{Proposition}
\newtheorem{remark}[thm]{Remark}
\newcommand{\R}{{\mathbb R}}
\renewcommand{\S}{{\mathbb S}}
\newcommand{\N}{{\mathbb N}}
\newcommand{\be}[1]{\begin{equation}\label{#1}}
\newcommand{\ee}{\end{equation}}
\renewcommand{\(}{\left(}
\renewcommand{\)}{\right)}
\newcommand{\eps}{\varepsilon}
\newcommand{\seq}[2]{({#1}_{#2})_{#2\in\N}}

\usepackage{color}

\renewcommand{\S}{\mathbb{S}}
\newcommand{\iS}[1]{\int_{\S^d}{#1}\;d\sigma}
\newcommand{\iRd}[1]{\int_{\R^d}{#1}\;dx}
\newcommand{\nrm}[2]{\|{#1}\|_{\L^{#2}(\S^d)}}
\newcommand{\nrmRd}[2]{\|{#1}\|_{\L^{#2}(\R^d)}}

\renewcommand{\H}{\mathrm H}
\renewcommand{\L}{\mathrm L}

\begin{document}
\title[Spectral estimates on the sphere]{Spectral estimates on the sphere}
\author[J.~Dolbeault , M.J.~Esteban, A.~Laptev]{Jean Dolbeault, Maria J.~Esteban, and Ari Laptev}
\address{J.~Dolbeault \& M.J.~Esteban: Ceremade CNRS UMR 7534, Universit\'e Paris-Dauphine, Place de Lattre de Tassigny, 75775 Paris C\'edex~16, France.
{\sl E-mail addresses:\/} \rm\textsf{dolbeaul@ceremade.dauphine.fr, esteban@ceremade.dauphine.fr}}
%\address{M.J.~Esteban: Ceremade, Universit\'e Paris-Dauphine, Place de Lattre de Tassigny, 75775 Paris C\'edex~16, France.
%\Email{esteban@ceremade.dauphine.fr}}
\address{A.~Laptev: Department of Mathematics, Imperial College London, Huxley Building, 180 Queen's Gate, London SW7 2AZ, UK.
{\sl E-mail address:\/} \rm\textsf{a.laptev@imperial.ac.uk}}
\date{\today}

\begin{abstract}
In this article we establish optimal estimates for the first eigenvalue of Schr\"odinger operators on the $d$-dimensional unit sphere. These estimates depend on $\L^p$ norms of the potential, or of its inverse, and are equivalent to interpolation inequalities on the sphere. We also characterize a semi-classical asymptotic regime and discuss how our estimates on the sphere differ from those on the Euclidean space.
\end{abstract}

\maketitle
\thispagestyle{empty}
%%%%%%%%%%%%%%%%%%%%%%%%%%%%%%%%%%%%%%%%%%%%%%%%%%%%%%%%%%%%%%%%%%%%%%
%\tableofcontents

\section{Introduction}\label{Sec:Intro}

Let $\Delta$ be the Laplace-Beltrami operator on the unit $d$-dimensional sphere $\S^d$. Our first result is concerned with the sharp estimate of the first negative eigenvalue $\lambda_1=\lambda_1(-\Delta-V)$ of the Schr\"odinger operator $-\Delta-V$ on $\S^d$ (with potential $-V$) in terms of $\L^p$-norms of $V$.

The literature on spectral estimates for the negative eigenvalues of Schr\"odinger operators on manifolds is limited. We can quote two papers of P.~Federbusch and O.S.~Rothaus, \cite{federbush1969partially,Rothaus-81}, which establish a link between logarithmic Sobolev inequalities and the ground state energy of Schr\"odinger operators. The Rozenbljum-Lieb-Cwikel inequality (case $\gamma=0$ with standard notations: see below) on manifolds has been studied in \cite[Section 5]{MR1454250}; we may also refer to \cite{MR0407909} for the semi-classical regime, and to \cite{MR2213030,MR2639180} for more recent results in this direction. In two articles (see \cite{MR121812,MR2879309}) on Lieb-Thirring type inequalities (also see \cite{MR2213030,MR2639180} for other results on manifolds), A.~Ilyin considers Schr\"odinger operators on unit spheres restricted to the space of functions orthogonal to constants and uses the original method of E.~Lieb and W.~Thirring in \cite{Lieb-Thirring76}. The exclusion of the zero mode of the Laplace-Beltrami operator results in \emph{semi-classical} estimates similar to those for negative eigenvalues of Schr\"odinger operators in Euclidean spaces.

The results in this paper are somewhat complementary. We show that if the $\L^p$-norm of~$V$ is smaller than an explicit value, then the first eigenvalue $\lambda_1(-\Delta-V)$ cannot satisfy the semi-classical inequality and thus it is impossible to obtain standard Lieb-Thirring type inequalities for the whole negative spectrum. However, we show that if the $\L^p$-norm of the potential is large then the first eigenvalue behaves semi-classically and the best constant in the inequality asymptotically coincides with the best constants $\L_{\gamma,d}^1$ of the corresponding inequality in the Euclidean space of same dimension (see below). In this regime the first eigenfunction is concentrated around some point on $\S^d$ and can be identified with an eigenfunction of the Schr\"odinger operator on the tangent space, up to a small error. In Appendix~\ref{Sec:AppendixA}, we illustrate the transition between the small $\L^p$-norm regime and the asymptotic, semi-classical regime by numerically computing the optimal estimates for the eigenvalue $\lambda_1(-\Delta -V)$ in terms of the norms $\nrm Vp$.

In order to formulate our first theorem let us introduce the measure $d\omega$ induced by Lebesgue's measure on $\S^d\subset\R^{d+1}$ and the uniform probability measure $d\sigma=d\omega/|\S^d|$ with $|\S^d|=\omega(\S^d)$. We shall denote by $\nrm\cdot q$ the quantity $\nrm uq=\big(\iS{|u|^q}\big){}^{1/q}$ for any $q>0$ (hence including in the case $q\in(0,1)$, for which $\nrm\cdot q$ is not anymore a norm, but only a quasi-norm). Because of the normalization of $d\sigma$, when making comparisons with corresponding results in the Euclidean space, we will need the constant
\[
\kappa_{q,d}:=|\S^d|^{1-\frac2q}\,.
\]
The well-known optimal constant $\L_{\gamma,d}^1$ in the one bound state Keller-Lieb-Thirring inequality is defined as follows: for any function $\phi$ on $\R^d$, if $\lambda_1(-\Delta-\phi)$ denotes the lowest negative eigenvalue of the Schr\"odinger operator $-\Delta-\phi$ (with potential $-\phi$) when it exists, and $0$ otherwise, we have
\be{Lgamma,d}
|\lambda_1(-\Delta-\phi)|^\gamma\le\L_{\gamma,d}^1\iRd{\phi_+^{\gamma+\frac d2}}\,,
\ee
provided $\gamma\ge0$ if $d\ge3$, $\gamma>0$ if $d=2$, and $\gamma\ge1/2$ if $d=1$. Notice that only the positive part $\phi_+$ of $\phi$ is involved in the right-hand side of the above inequality. Assuming that $\gamma>1-d/2$ if $d=1$ or $2$, we shall consider the exponents
\[
q=2\,\frac{2\,\gamma+d}{2\,\gamma+d-2}\quad\mbox{and}\quad p=\frac q{q-2}=\gamma+\frac d2\,,
\]
which are therefore such that $2<q=\frac{2\,p}{p-1}\le2^*$ with $2^*:=\frac{2\,d}{d-2}$ if $d\ge3$, and $q=\frac{2\,p}{p-1}\in(2,+\infty)$ if $d=1$ or $2$. To simplify notations, we adopt the convention $2^*:=\infty$ if $d=1$ or $2$. It is also convenient to introduce the notation
\[
\alpha_*:=\frac14\,d\,(d-2)\;.
\]
In Section \ref{Sec:Interpolationq>2} we shall prove the following result.
%---------------------------------------------------------------------
\begin{thm}\label{Thm1} Let $d\ge1$, $p\in\big(\max\{1,d/2\},+\infty\big)$. Then there exists a convex increasing function $\alpha:\R^+\to\R^+$ with $\alpha(\mu)=\mu$ for any $\mu\in\big[0,\tfrac d2\,(p-1)\big]$ and $\alpha(\mu)>\mu$ for any $\mu\in\big(\tfrac d2\,(p-1),+\infty\big)$, such that
\be{T11}
|\lambda_1(-\Delta-V)|\le\alpha\big(\nrm Vp\big)
\ee
for any nonnegative $V\in\L^p(\S^d)$. Moreover, for large values of $\mu$, we have
\[\label{T12}
\alpha(\mu)^{p-\frac d2}=\L_{p-\frac d2,d}^1\,(\kappa_{q,d}\,\mu)^p\,(1+o(1))\,.
\]
The estimate \eqref{T11} is optimal in the sense that there exists a nonnegative function $V$ such that $\mu=\nrm Vp$ and $|\lambda_1(-\Delta-V)|=\alpha(\mu)$ for any $\mu\in\big(\frac d2\,(p-1),+\infty\big)$. If $\mu\le\frac d2\,(p-1)$, equality in~\eqref{T11} is achieved by constant potentials.

If $p=d/2$ and $d\ge3$, then \eqref{T11} is satisfied with $\alpha(\mu)=\mu$ only for $\mu\in[0,\alpha_*]$. If $d=p=1$, then \eqref{T11} is also satisfied for some nonnegative, convex function $\alpha$ on $\R^+$ such that $\mu\le\alpha(\mu)\le\mu+\pi^2\,\mu^2$ for any $\mu\in(0,+\infty)$, equality in \eqref{T11} is achieved and $\alpha(\mu)=\pi^2\,\mu^2(1+o(1))$ as $\mu\to+\infty$.
\end{thm}
%---------------------------------------------------------------------
Since $\lambda_1(-\Delta-V)$ is nonpositive for any nonnegative, nontrivial $V$, inequality~\eqref{T11} is a lower estimate. We have indeed found that
\[
0\ge\lambda_1(-\Delta-V)\ge-\,\alpha\big(\nrm Vp\big)\,.
\]
If $V$ changes sign, the above inequality still holds if $V$ is replaced by the positive part $V_+$ of $V$, provided the lowest eigenvalue is negative. We can then write
\be{KLT}
|\lambda_1(-\Delta-V)|\le\alpha\(\nrm{V_+}p\)\quad\forall\,V\in\L^p(\S^d)\,.
\ee
The expression of $\L_{\gamma,d}^1$ is not explicit (except in the case $d=1$: see \cite[p.~290]{Lieb-Thirring76}) but can be given in terms of an optimal constant in some Gagliardo-Nirenberg-Sobolev inequality (see~\cite{Lieb-Thirring76}, and \eqref{Ineq:GNS}--\eqref{Eq:KLT-GNS} below in Section~\ref{Sec:GNS-Euclidean}). In case $d=p=1$, notice that $\L_{1/2,1}^1=1/2$ (see Appendix~\ref{Sec:AsympGNS}) and $\kappa_{\infty,1}=2\pi$ so that our formula in the asymptotic regime $\mu\to+\infty$ is consistent with the other cases.

The reader is invited to check that Theorem~\ref{Thm1} can be reformulated in a more standard language of spectral theory as follows. We recall that $\gamma=p-d/2$ and that $d\omega$ is the standard measure induced on the unit sphere $\S^d$ by Lebesgue's measure on $\R^{d+1}$.
%---------------------------------------------------------------------
\begin{cor}\label{cor:negpot} Let $d\ge1$ and consider a nonnegative function $V$. For $\mu=\nrm V{\gamma+\frac d2}$ large, we have
\be{Lgammap1Largemu}
|\lambda_1(-\Delta-V)|^\gamma\lesssim\L_{\gamma,d}^1\,\int_{\S^d}V^{\gamma+\frac d2}\;d\omega
\ee
if either $\gamma > \max\{0,1-d/2\}$ or $\gamma=1/2$ and $d=1$. However, if $\mu=\nrm V{\gamma+\frac d2}\le\frac14\,d\,(2\,\gamma+d-2)$, then we have
\be{Lgammap1Smallmu}
|\lambda_1(-\Delta-V)|^{\gamma+\frac d2}\le\int_{\S^d}V^{\gamma+\frac d2}\;d\omega
\ee
for any $\gamma\ge\max\{0,1-d/2\}$ and this estimate is optimal.\end{cor}
%---------------------------------------------------------------------
\noindent Here the notation $f\lesssim g$ as $\mu\to+\infty$ means that $f\le c(\mu)\,g$ with $\lim_{\mu\to\infty}c(\mu)=1$. The limit case $\gamma=\max\{0,1-d/2\}$ in \eqref{Lgammap1Smallmu} is covered by approximations. We may also notice that optimality in \eqref{Lgammap1Smallmu} is achieved by constant potentials. Let us give some details.

If we consider a sequence of constant functions $\seq Vn$ uniformly converging towards $0$, for instance $V_n=1/n$, then we get that
\[
\lim_{n\to\infty}\frac{|\lambda_1(-\Delta-V_n)|^\gamma}{\int_{\S^d}V_n^{\gamma+\frac d2}\;d\omega}=+\infty
\]
which clearly forbids the possibility of an inequality of the same type as~\eqref{Lgammap1Largemu} for small values of $\int_{\S^d}V^{\gamma+\frac d2}\;d\omega$. This is however compatible with the results of A.~Ilyin in dimension $d=2$. In \cite[Theorem 2.1]{MR2879309}, the author states that if $P$ is the orthogonal projection defined by $P\,u:=u-\int_{\S^2} u\;d\omega$, then the negative eigenvalues $\lambda_k (P\,(-\Delta -V)\,P)$ satisfy the semi-classical inequality
\[
\sum_k|\lambda_k(P\,(-\Delta -V)\,P)|\le\frac38\int_{\S^2}V^2\;d\omega\,.
\]
Another way of seeing that inequalities like~\eqref{Lgammap1Largemu} are incompatible with small potentials is based on the following observation. Inequality \eqref{Lgammap1Smallmu} shows that
\[
|\lambda_1(-\Delta -V)|\le\Big(\int_{\S^2}V^2\,d\omega\Big)^{1/2}
\]
if the $\L^2$-norm of $V$ is smaller than $1$. Since such an inequality is sharp, the semi-classical Lieb-Thirring inequalities for the Schr\"odinger operator on the sphere~$\S^2$ are therefore impossible for small potentials and can be achieved only in a semi-classical asymptotic regime, that is, when the norm $\|V\|_{\L^2(\S^2)}$ is large.

\medskip Our second main result is concerned with the estimates from below for the first eigenvalue of Schr\"odinger operators with positive potentials. In this case, by analogy with \eqref{Lgamma,d}, it is convenient to introduce the constant $\L_{-\gamma,d}^1$ with $\gamma>d/2$ which is the optimal constant in the inequality:
\be{L-gamma,d}
\lambda_1(-\Delta+\phi)^{-\gamma}\le\L_{-\gamma,d}^1\iRd{\phi^{\frac d2-\gamma}}\,,
\ee
where $\phi$ is any positive potential on $\R^d$ and $\lambda_1(-\Delta+\phi)$ denotes the lowest positive eigenvalue if it exists, or $+\infty$ otherwise. Inequality~\eqref{L-gamma,d} is less standard than \eqref{Lgamma,d}: we refer to \cite[Theorem 12]{MR2253013} for a statement and a proof. As in Theorem~\ref{Thm1}, we shall also introduce exponents $p$ and $q$ such that
\[
q=2\,\frac{2\,\gamma-d}{2\,\gamma-d+2}\quad\mbox{and}\quad p=\frac q{2-q}=\gamma-\frac d2\,,
\]
so that $p$ (resp.~$q=\frac{2\,p}{p+1}$) takes arbitrary values in $(0,+\infty)$ (resp.~$(0,2)$). With these notations, we have the counterpart of Theorem~\ref{Thm1} in the case of positive potentials.
%---------------------------------------------------------------------
\begin{thm}\label{Thm2} Let $d\ge1$, $p\in(0,+\infty)$. There exists a concave increasing function $\nu:\R^+\to\R^+$ with $\nu(\beta)=\beta$ for any $\beta\in\big[0,\tfrac d2\,(p+1)\big]$ if $p>1$, $\nu(\beta)\le\beta$ for any $\beta>0$ and $\nu(\beta)<\beta$ for any $\beta\in\big(\tfrac d2\,(p+1),+\infty\big)$, such that
\be{T31}
\lambda_1(-\Delta+W)\ge\nu\big(\beta\big)\quad\mbox{with}\quad\beta=\nrm{W^{-1}}p^{-1}\,,
\ee
for any positive potential $W$ such that $W^{-1}\in\L^p(\S^d)$. Moreover, for large values of $\beta$, we have
\[\label{T32}
\nu(\beta)^{-\,(p+\frac d2)}\lesssim\L_{-(p+\frac d2),d}^1\,\(\kappa_{q,d}\,\beta\)^{-p}\,.
\]
The estimate \eqref{T31} is optimal in the sense that there exists a nonnegative potential $W$ such that $\beta^{-1}=\nrm{W^{-1}}p$ and $\lambda_1(-\Delta+W)=\nu(\beta)$ for any positive $\beta$ and $p$. If $\beta\le\frac d2\,(p+1)$ and $p>1$, equality in~\eqref{T31} is achieved by constant potentials.
\end{thm}
%---------------------------------------------------------------------
Again the expression of $\L_{-\gamma,d}^1$ is not explicit when $d\ge2$ but can be given in terms of an optimal constant in some Gagliardo-Nirenberg-Sobolev inequality (see~\cite{MR2253013}, and \eqref{GNS2}--\eqref{Eq:KLT-GNS2} below in Section~\ref{Sec:Confining}).

We can rewrite Theorem~\ref{Thm2} in terms of $\gamma=p+d/2$ and explicit integrals involving $W$.
%---------------------------------------------------------------------
\begin{cor}\label{cor:posit} Let $d\ge1$ and $\gamma>d/2$. For $\beta=\nrm{W^{-1}}{\gamma-\frac d2}^{-1}$ large, we have
\[\label{Lminusgammap1Largemu}
\big(\lambda_1(-\Delta+W)\big)^{-\gamma}\lesssim\L_{-\gamma,d}^1\,\int_{\S^d}W^{\frac d2-\gamma}\;d\omega\,.
\]
However, if $\gamma\ge\frac d2+1$ and if $\beta=\nrm{W^{-1}}{\gamma-\frac d2}^{-1}\le\frac14\,d\,(2\,\gamma-d+2)$, then we have
\[
\big(\lambda_1(-\Delta+W)\big)^{\frac d2-\gamma}\le\int_{\S^d}W^{\frac d2-\gamma}\;d\omega\,,
\]
and this estimate is optimal.\end{cor}
%---------------------------------------------------------------------

\medskip This paper is organized as follows. Section \ref{Sec:Interpolationq>2} contains various results on interpolation inequalities; the most important one for our purpose is stated in Lemma~\ref{Lem:q>2}. Theorem \ref{Thm1}, Corollary~\ref{cor:negpot} and various spectral estimates for Schr\"odinger operators with negative potentials are established in Section \ref{Sec:Spectral}. Section \ref{Sec:Confining} deals with the case of positive potentials and contains the proofs of Theorem \ref{Thm2} and Corollary \ref{cor:posit}. Section \ref{Sec:Conclusion} is devoted to the threshold case ($q=2$, that is, $p$, $\gamma\to+\infty$) of exponential estimates for eigenvalues or, in terms of interpolation inequalities, to logarithmic Sobolev inequalities. Finally numerical and technical results have been collected in two appendices.

%%%%%%%%%%%%%%%%%%%%%%%%%%%%%%%%%%%%%%%%%%%%%%%%%%%%%%%%%%%%%%%%%%%%%%
%%%%%%%%%%%%%%%%%%%%%%%%%%%%%%%%%%%%%%%%%%%%%%%%%%%%%%%%%%%%%%%%%%%%%%
\section{Interpolation inequalities and consequences for negative potentials}\label{Sec:Interpolationq>2}

%%%%%%%%%%%%%%%%%%%%%%%%%%%%%%%%%%%%%%%%%%%%%%%%%%%%%%%%%%%%%%%%%%%%%%
\subsection{Inequalities in the Euclidean space}\label{Sec:GNS-Euclidean}

Let us start by some considerations on inequalities in the Euclidean space, which play a crucial role in the semi-classical regime.

We recall that we denote by $2^*$ the Sobolev critical exponent $\frac{2d}{d-2}$ if $d\ge 3$ and consider Sobolev's inequality on $\R^d$, $d\ge3$,
\be{Ineq:Sobolev}
\nrmRd v{2^*}^2\le\mathsf S_d\,\nrmRd{\nabla v}2^2\quad\forall\,v\in\mathcal D^{1,2}(\R^d)
\ee
where $\mathsf S_d$ is the optimal constant and $\mathcal D^{1,2}(\R^d)$ is the Beppo-Levi space obtained by completion of smooth compacty supported functions with respect to the norm $v\mapsto \nrmRd{\nabla v}2$. See Appendix \ref{Sec:Sobolev} for details and comments on the expression of $\mathsf S_d$.

Assume now that $d\ge1$ and recall that $2^*=+\infty$ if $d=1$ or $2$. In the subcritical case, that is, $q\in(2,2^*)$, let
\[
\mathsf K_{q,d}:=\inf_{v\in\H^1(\R^d)\setminus\{0\}}\frac{\nrmRd{\nabla v}2^2+\nrmRd v2^2}{\nrmRd vq^2}
\]
be the optimal constant in the Gagliardo-Nirenberg-Sobolev inequality
\be{Ineq:GNS}
\mathsf K_{q,d}\,\nrmRd vq^2\le\nrmRd{\nabla v}2^2+\nrmRd v2^2\quad\forall\,v\in\H^1(\R^d)\,.
\ee
The optimal constant $\L_{\gamma,d}^1$ in the one bound state Keller-Lieb-Thirring inequality is such that
\be{Eq:KLT-GNS}
\L_{\gamma,d}^1:=\(\mathsf K_{q,d}\)^{-\,p}\quad\mbox{with}\quad p=\gamma+\frac d2\,,\quad q=2\,\frac{2\,\gamma+d}{2\,\gamma+d-2}\,.
\ee
See Appendix~\ref{Sec:GNS-KLT} for a proof and references, and \cite{Lieb-Thirring76} for a detailed discussion. Also see \cite[Appendix A. Numerical studies, by~J.F. Barnes]{Lieb-Thirring76} for numerical values of $\mathsf K_{q,d}$.

We shall also define the exponent
\[
\vartheta:=d\,\frac{q-2}{2\,q}
\]
which plays an important role in the scale invariant form of the Gagliardo-Nirenberg-Sobolev interpolation inequalities associated to $\mathsf K_{q,d}$: see Appendix~\ref{Sec:GNS} for details.

%%%%%%%%%%%%%%%%%%%%%%%%%%%%%%%%%%%%%%%%%%%%%%%%%%%%%%%%%%%%%%%%%%%%%%
\subsection{Interpolation inequalities on the sphere}\label{Sec:InterpSphere}

Using the inverse stereographic projection (see Appendix~\ref{Sec:Stereographic}), it is possible to relate interpolation inequalities on $\R^d$ with interpolation inequalities on $\S^d$. In this section we consider the case of the sphere. Notice that $\alpha_*=d/(q-2)$ when $q=2^*=2\,d/(d-2)$, $d\ge3$.
%---------------------------------------------------------------------
\begin{lem}\label{Lem:q>2} Let $q\in(2,2^*)$. Then there exists a concave increasing function $\mu:\R^+\to\R^+$ with the following properties
\[
\mu(\alpha)=\alpha\quad\forall\,\alpha\in\big[0,\tfrac d{q-2}\big]\quad\mbox{and}\quad\mu(\alpha)<\alpha\quad\forall\,\alpha\in\big(\tfrac d{q-2},+\infty\big)\,,
\]
\[
\mu(\alpha)=\frac{\mathsf K_{q,d}}{\kappa_{q,d}}\,\alpha^{1-\vartheta}\,(1+o(1))\quad\mbox{as}\quad\alpha\to+\infty\,,
\]
such that
\be{InterpSphere}
\nrm{\nabla u}2^2+\alpha\,\nrm u2^2\ge\mu(\alpha)\,\nrm uq^2\quad\forall\,u\in\H^1(\S^d)\,.
\ee

\noindent If $d\ge3$ and $q=2^*$, the inequality also holds for any $\alpha>0$ with $\mu(\alpha)=\min\left\{\alpha,\alpha_*\right\}$.\end{lem}
%---------------------------------------------------------------------
The remainder of this section is mostly devoted to the proof of Lemma~\ref{Lem:q>2}. A fundamental tool is a \emph{rigidity} result proved by M.-F.~Bidaut-V\'eron and L.~V\'eron in~\cite[Theorem 6.1]{BV-V} for $q>2$, which goes as follows. Any positive solution of
\be{Eqn:BV-V}
-\,\Delta f+\alpha\,f=f^{q-1}
\ee
has a unique solution $f\equiv\alpha^{1/(q-2)}$ for any $0<\alpha\le d/(q-2)$. A straightforward consequence of this rigidity result is the following interpolation inequality (see ~\cite[Corollary 6.2]{BV-V}):
\be{Ineq:Interpolation}
\iS{|\nabla u|^2}\ge\frac d{q-2}\left[\(\iS{|u|^q}\)^{2/q}-\iS{|u|^2}\right]\quad\forall\,u\in\H^1(\S^d,d\sigma)\,.
\ee
Inequality~\eqref{Ineq:Interpolation} holds for any $q\in[1,2)\cup(2,2^*]$ if $d\ge 3$ and for any $q\in[1,2)\cup(2,\infty)$ if $d=1$ or $2$. An alternative proof of~\eqref{Ineq:Interpolation} has been established in \cite{MR1230930} for $q>2$ using previous results by E.~Lieb in~\cite{MR717827} and the Funk-Hecke formula (see \cite{45.0702.01,46.0632.02}). The whole range $p\in[1,2)\cup(2, 2^*)$ was covered in the case of the ultraspherical operator in \cite{MR2564058,MR2641798}. Also see \cite{MR1412446,MR1813804} for the \emph{carr\'e du champ} method, and \cite{DEKL} for an elementary proof. Inequality~\eqref{Ineq:Interpolation} is \emph{tight} as defined by D.~Bakry in \cite[Section~2]{MR2213477}, in the sense that equality is achieved only by constants.

%---------------------------------------------------------------------
\begin{remark} Inequality~\eqref{Ineq:Interpolation} is equivalent to
\[
\inf_{u\in\H^1(\S^d)\setminus\{0\}}\frac{(q-2)\,\nrm{\nabla u}2^2}{\nrm uq^2-\nrm u2^2}=d\,.
\]
Although we will not make use of them in this paper, we may  notice that the following properties hold true:
\begin{enumerate}
\item[(i)] If $q<2^*$, the above infimum is not achieved in $\H^1(\S^d)\setminus\{0\}$ but
\[
\lim_{\eps\to0_+}\frac{(q-2)\,\nrm{\nabla u_\eps}2^2}{\nrm{u_\eps}q^2-\nrm{u_\eps}2^2}=d
\]
if $u_\eps:=1+\eps\,\varphi$, where $\varphi$ is a non-trivial eigenfunction of the Laplace-Beltrami operator corresponding the first nonzero eigenvalue (see below Section~\ref{Sec:PropertiesAlpha(mu)}).
\item[(ii)] If $q=2^*$, $d\ge3$, there are non-trivial optimal functions for ~\eqref{Ineq:Interpolation}, due to the conformal invariance. Alternatively, these solutions can be constructed from the family of Aubin-Talenti optimal functions for Sobolev's inequality, using the inverse stereographic projection.
\item[(iii)] If $\alpha>\alpha_*$ and $q=2^*$, $d\ge3$, there are no optimal functions for \eqref{InterpSphere}, since otherwise $\alpha\mapsto\mu(\alpha)$ would not be constant on $(\alpha_*,\alpha)$: see Proposition~\ref{Prop:Critical} below.
\end{enumerate}
\end{remark}
%---------------------------------------------------------------------

%%%%%%%%%%%%%%%%%%%%%%%%%%%%%%%%%%%%%%%%%%%%%%%%%%%%%%%%%%%%%%%%%%%%%%
\subsection{Properties of the function \texorpdfstring{$\alpha\mapsto\mu(\alpha)$ in the subcritical case}{alpha mapsto mu(alpha)}}\label{Sec:PropertiesAlpha(mu)}

Assume that $q\in(2,2^*)$. For any $\alpha>0$, consider
\[
\mathcal Q_\alpha[u]:=\frac{\nrm{\nabla u}2^2+\alpha\,\nrm u2^2}{\nrm uq^2}\quad\forall\,u\in\H^1(\S^d,d\sigma)\,.
\]
It is a standard result of the calculus of variations that
\[
\inf_{\begin{array}{c}u\in\H^1(\S^d,d\sigma)\\ \iS{|u|^q}=1\end{array}}\mathcal Q_\alpha[u]:=\mu(\alpha)
\]
is achieved by a minimizer $u\in\H^1(\S^d,d\sigma)$ which solves the Euler-Lagrange equations
\be{EL}
-\,\Delta u+\alpha\,u-\mu(\alpha)\,u^{q-1}=0\,.
\ee
Indeed we know that there is a Lagrange multiplier associated to the constraint $\iS{|u|^q}=1$, and multiplying \eqref{EL} by $u$ and integrating on $\S^d$, we can identify it with $\mu(\alpha)$. As a corollary, we have shown that~\eqref{InterpSphere} holds. The fact that the Lagrange multiplier can be identified so easily is a consequence of the fact that all terms in~\eqref{InterpSphere} are two-homogeneous.

\medskip We can now list some basic properties of the function $\alpha\mapsto\mu(\alpha)$.
\begin{enumerate}
\item For any $\alpha>0$, $\mu(\alpha)$ is positive, since the infimum is achieved by a nonnegative function $u$ and $u=0$ is incompatible with the constraint $\iS{|u|^q}=1$. By taking a constant test function, we see that $\mu(\alpha)\le \alpha$, for all $\alpha>0$. The function $\alpha\mapsto\mu(\alpha)$ is monotone nondecreasing since for a given $u\in\H^1(\S^d,d\sigma)\setminus\{0\}$, the function $\alpha\mapsto\mathcal Q_\alpha[u]$ is monotone increasing. It is actually strictly monotone. Indeed if $\mu(\alpha_1)=\mu(\alpha_2)$ with $\alpha_1<\alpha_2$, then one can notice that $\mathcal Q_{\alpha_1}[u_2]<\mu(\alpha_1)$ if~$u_2$ is a minimizer of $\mathcal Q_{\alpha_2}$ satisfying the constraint $\iS{|u_2|^q}=1$, which provides an obvious contradiction.
\item We have
\[
\mu(\alpha)=\alpha\quad\forall\,\alpha\in\(0,\tfrac d{q-2}\right]\,.
\]
Indeed, if $u$ is a solution of \eqref{EL}, then $f=\mu(\alpha)^{1/(q-2)}\,u$ solves \eqref{Eqn:BV-V} and is therefore a constant function if $\alpha\le d/(q-2)$ according to \cite[Theorem~6.1]{BV-V}, and so is $u$ as well. Because of the normalization constraint $\nrm uq=1$, we get that $u=1$, which proves the statement.

On the contrary, we have
\[
\mu(\alpha)<\alpha\quad\forall\,\alpha>\frac d{q-2}\,.
\]
Let us prove this. Let $\varphi$ be a non-trivial eigenfunction of the Laplace-Beltrami operator corresponding the first nonzero eigenvalue:
\[
-\,\Delta\varphi=d\,\varphi\,.
\]
If $x=(x_1,x_2,...x_d,z)$ are cartesian coordinates of $x\in\R^{d+1}$ so that $\S^d\subset\R^{d+1}$ is characterized by the condition $\sum_{i=1}^d x_i^2+z^2=1$, then a simple choice of such a function $\varphi$ is $\varphi(x)=z$. By orthogonality with respect to the constants, we know that $\iS\varphi=0$. We may now Taylor expand $\mathcal Q_\alpha$ around $u=1$ by considering $u=1+\eps\,\varphi$ as $\eps\to0$ and obtain that
\[
\mu(\alpha)\le \mathcal Q_\alpha[1+\eps\,\varphi]=\frac{(d+\alpha)\,\eps^2\iS{|\varphi|^2} +\alpha}{\(\iS{|1+\eps\,\varphi|^q}\)^{2/q}}=\alpha+\big[d+\alpha\,(2-q)\big]\,\eps^2\iS{|\varphi|^2}+o(\eps^2)\,.
\]
By taking $\eps$ small enough, we get $\mu(\alpha)<\alpha$ for all $\alpha>d/(q-2)$. Optimizing on the value of $\eps>0$ (not necessarily small) provides an interesting test function: see Section~\ref{Sec:RefinedUpper}.
\item The function $\alpha\mapsto\mu(\alpha)$ is concave, because it is the minimum of a family of affine functions.
\end{enumerate}

%%%%%%%%%%%%%%%%%%%%%%%%%%%%%%%%%%%%%%%%%%%%%%%%%%%%%%%%%%%%%%%%%%%%%%
\subsection{More estimates on the function \texorpdfstring{$\alpha\mapsto\mu(\alpha)$}{alpha mapsto mu(alpha)}}\label{Sec:EstimatesAlpha(mu)}

We first consider the critical case $q=2^*$, $d\ge3$. As in the subcritical case $q<2^*$, we have $\mu(\alpha)=\alpha$ for $\alpha\le\alpha^*$. For $\alpha>\alpha^*$, the function $\alpha\mapsto\mu(\alpha)$ is constant:
%---------------------------------------------------------------------
\begin{prop}\label{Prop:Critical} With the notations of Lemma~\ref{Lem:q>2}, if $d\ge3$ and $q=2^*$, then
\[
\mu(\alpha)=\alpha_*\quad\forall\,\alpha>\alpha_*=\frac d{q-2}=\frac 14\,d\,(d-2)\,.
\]
\end{prop}
%---------------------------------------------------------------------
\begin{proof} Consider the Aubin-Talenti optimal functions for Sobolev's inequality and more specifically, let us choose the functions
\[
v_\eps(x):=\(\tfrac\eps{\eps^2+|x|^2}\)^\frac{d-2}2\quad\forall\,x\in\R^d\,,\quad\forall\,\eps>0\,,
\]
which are such that $\nrmRd{v_\eps}{2^*}=\nrmRd{v_1}{2^*}$ is independent of $\eps$. With standard notations (see Appendix~\ref{Sec:Stereographic}), let $\mathrm N\in\S^d$ be the North Pole. Using the stereographic projection $\Sigma$, \emph{i.e.}~for the functions defined for any $y\in\S^d\setminus\{\mathrm N\}$ by
\[
u_\eps(y)=\(\tfrac{|x|^2+1}2\)^\frac{d-2}2v_\eps(x)\quad\mbox{with}\quad x=\Sigma(y)\,,
\]
we find that $\nrm{u_\eps}{2^*}=\nrmRd{v_1}{2^*}$ for any $\eps>0$, so that
\[
\mu(\alpha)\le\mathcal Q_\alpha[u_\eps]=\frac{\nrmRd{\nabla v_\eps}2^2+(\alpha-\alpha_*)\,\iRd{|v_\eps|^2\,\big(\tfrac2{1+|x|^2}\big)^2}}{\kappa_{2^*,d}\,\nrmRd{v_\eps}{2^*}^2}=\alpha_*+4\,|\S^d|^{1-\frac 2d}\,(\alpha-\alpha_*)\,\frac{\delta(d,\eps)}{\nrmRd{v_1}{2^*}^2}
\]
where we have used the fact that $\kappa_{2^*,d}\,\mathsf S_d=1/\alpha_*$ (see Appendix~\ref{Sec:Sobolev}) and
\[
\delta(d,\eps):=\int_0^\infty\(\frac\eps{\eps^2+r^2}\)^{d-2}\frac{r^{d-1}}{(1+r^2)^2}\;dr=\eps^2\int_0^\infty\(\frac1{1+s^2}\)^{d-2}\frac{s^{d-1}}{(1+\eps^2\,s^2)^2}\;ds\,.
\]
One can check that $\lim_{\eps\to0_+}\delta(d,\eps)=0$ since
\[
\delta(d,\eps)\le\eps^2\int_0^\infty\frac{s^{d-1}}{(1+s^2)^{d-2}}\;ds\quad\mbox{if}\quad d\ge5\quad\mbox{and}\quad\delta(d,\eps)\le \,\eps\,c_d\int_0^{+\infty}\frac{ds}{(1+s^2)^2}\quad\mbox{if}\quad d=3\;\mbox{or}\;4\,,
\]
with $c_3=1$ and $c_4=3\,\sqrt3/16$.\end{proof}

The next step is devoted to a lower estimate for the function $\alpha\mapsto \mu(\alpha)$ in the subcritical case, which shows that $\lim_{\alpha\to+\infty}\mu(\alpha)=+\infty$ in contrast with the critical case.
%---------------------------------------------------------------------
\begin{prop}\label{Cor:CriticalSubCritical} With the notations of Lemma~\ref{Lem:q>2}, if $d\ge3$ and $q\in(2,2^*)$, then for any $\alpha>\alpha_*$ we have
\[\label{criticalestimatebelow}
\alpha>\mu(\alpha)\ge\alpha_*^\vartheta\,\alpha^{1-\vartheta}\,,
\]
with $\vartheta=d\,\frac{q-2}{2\,q}$. For every $s\in (2,2^*)$ if $d\ge3$, or every $s\in(2,+\infty)$ if $d=1$ or $2$, such that $s>q$, we also have that
\[\label{notcriticalestimatebelow}
\alpha>\mu(\alpha)\ge \big(\tfrac d{s-2}\big)^\theta\,\alpha^{1-\theta}\,,
\]
for any $\alpha>d/(s-2)$ and $\theta=\theta(s,q,d):=\frac{s\,(q-2)}{q\,(s-2)}$.
\end{prop}
%---------------------------------------------------------------------
\begin{proof} The first case can be seen as a limit case of the second one as $s\to2^*$ and $\vartheta=\theta(2^*,q,d)$. Using H\"older's inequality, we can estimate $\nrm uq$ by
\[
\nrm uq\leq\nrm us^\theta\,\nrm u2^{1-\theta}
\]
and get the result using
\[
\mathcal Q_\alpha[u]\ge\(\frac{\nrm{\nabla u}2^2+\alpha\,\nrm u2^2}{\nrm us^2}\)^\theta\,\(\frac{\nrm{\nabla u}2^2+\alpha\,\nrm u2^2}{\nrm u2^2}\)^{1-\theta}\ge\big(\tfrac d{s-2}\big)^\theta\,\alpha^{1-\theta}\,.
\]
\end{proof}

%---------------------------------------------------------------------
\begin{prop}\label{prop-upper} With the notations of Lemma~\ref{Lem:q>2}, for every $q\in (2,2^*)$ we have
\[
\limsup_{\alpha\to+\infty}\alpha^{\vartheta-1}\mu(\alpha)\le\frac{\mathsf K_{q,d}}{\kappa_{q,d}}\,.
\]
\end{prop}
%---------------------------------------------------------------------
\begin{proof} Let $v$ be an optimal function for $\mathsf K_{q,d}$ and define for any $x\in\R^d$ the function
\[
v_\alpha(x):=v\(2\,\sqrt{\alpha-\alpha_*}\,x\)
\]
with $\alpha_*=\frac 14\,d\,(d-2)$ and $\alpha>\alpha_*$, so that
\[
\iRd{|\nabla v_\alpha|^2} =2^{2-d}\,(\alpha-\alpha_*)^{1-\frac{d}{2}} \iRd{|\nabla v|^2}\,,
\]
\[
\iRd{|v_\alpha|^q\,\Big(\tfrac2{1+|x|^2}\Big)^{d-(d-2)\frac q2}} =2^{-(d-2)\frac q2}\,(\alpha-\alpha_*)^{-\frac d2}\iRd{|v|^q\,\big(1+\tfrac{|x|^2}{4\,(\alpha-\alpha_*)}\big)^{-d+(d-2)\frac q2}}\,.
\]
Now we observe that the function $u_\alpha(y):=\big(\tfrac{|x|^2+1}2\big)^{(d-2)/2}\,v_\alpha(x)$, where $y=\Sigma^{-1}(x)$ and $\Sigma$ is the stereographic projection (see Appendix~\ref{Sec:Stereographic}), is such that
\[
\mathcal Q_\alpha[u_\alpha]=\frac1{\kappa_{q,d}}\,\frac{\iRd{|\nabla v_\alpha|^2}+(\alpha-\alpha_*)\iRd{|v_\alpha|^2 \,\Big(\tfrac2{1+|x|^2}\Big)^2}}{\Big[\iRd{|v_\alpha|^q\,\Big(\tfrac2{1+|x|^2}\Big)^{d-(d-2)\frac q2}}\Big]^\frac 2q}\,.
\]
Passing to the limit as $\alpha\to+\infty$, we get
\[
\lim_{\alpha\to+\infty}\iRd{|v|^q\,\big(1+\tfrac{|x|^2}{4\,(\alpha-\alpha_*)}\big)^{-d+(d-2)\frac q2}}=\iRd{|v|^q}
\]
by Lebesgue's theorem of dominated convergence. The limit also holds with $q$ replaced by $2$. This proves that
\[
\mathcal Q_\alpha[u_\alpha]=(\alpha-\alpha_*)^{1-\frac{d}2+\frac{d}q}\,\(\frac{\mathsf K_{q,d}}{\kappa_{q,d}}+o(1)\)\quad\mbox{as}\quad\alpha\to+\infty
\]
which concludes the proof because $\vartheta=d\,(q-2)/(2\,q)$.
\end{proof}

%%%%%%%%%%%%%%%%%%%%%%%%%%%%%%%%%%%%%%%%%%%%%%%%%%%%%%%%%%%%%%%%%%%%%%
\subsection{The semi-classical regime: behavior of the function \texorpdfstring{$\alpha\mapsto\mu(\alpha)$ as $\alpha\to+\infty$}{alpha mapsto mu(alpha) as alpha to infty}}\label{Sec:EstimatesMu(alpha)}

Assume that $q\in(2,2^*)$. If we combine the results of Propositions~\ref{Cor:CriticalSubCritical} and~\ref{prop-upper}, we know that $\mu(\alpha)\sim\alpha^{1-\vartheta}$ as $\alpha\to+\infty$ if $d\ge3$. If $d=1$ or $2$, we know that $\lim_{\alpha\to+\infty}\mu(\alpha)=+\infty$ with a growth at least equivalent to $\alpha^{2/q-\eps}$ with $\eps>0$, arbitrarily small, according to Proposition~\ref{Cor:CriticalSubCritical}, and at most equivalent to $\alpha^{1-\vartheta}$ by Proposition~\ref{prop-upper}. To complete the proof of Lemma~\ref{Lem:q>2}, it remains to determine the precise behavior of $\mu(\alpha)$ as $\alpha\to+\infty$.
%---------------------------------------------------------------------
\begin{prop}\label{prop-asympt}
With the notations of Lemma~\ref{Lem:q>2}, for every $q\in (2,2^*)$, with $\vartheta=d\,\frac{q-2}{2\,q}$ we have
\[
\mu(\alpha)=\frac{\mathsf K_{q,d}}{\kappa_{q,d}}\,\alpha^{1-\vartheta}(1+o(1))\quad\mbox{as}\quad\alpha\to+\infty\,.
\]
\end{prop}
%---------------------------------------------------------------------
\begin{proof} Suppose by contradiction that there is a positive constant~$\eta$ and a sequence $\seq\alpha n$ such that $\lim_{n\to+\infty}\alpha_n=+\infty$ and
\be{asumption-33}
\lim_{n\to+\infty}\alpha_n^{\vartheta-1}\mu(\alpha_n)\le\frac{\mathsf K_{q,d}}{\kappa_{q,d}}-\eta\,.
\ee
Consider a sequence $\seq un$ of functions in $\H^1(\S^d)$ such that $\mathcal Q_{\alpha_n}[u_n]=\mu(\alpha_n)$ and $\nrm{u_n}q=1$ for any $n\in\N$. From~\eqref{asumption-33}, we know that
\[
\alpha_n\,\nrm{u_n}2^2\le\mathcal Q_{\alpha_n}[u_n]=\mu(\alpha_n)\le\alpha_n^{1-\vartheta}\,\(\frac{\mathsf K_{q,d}}{\kappa_{q,d}}-\eta\)(1+o(1))\quad\mbox{as}\quad n\to+\infty\,
\]
that is
\[
\limsup_{n\to+\infty}\alpha_n^\vartheta\,\nrm{u_n}2^2\le\frac{\mathsf K_{q,d}}{\kappa_{q,d}}-\eta\,.
\]
The normalization $\nrm{u_n}q=1$ for any $n\in\N$ and the limit $\lim_{n\to+\infty}\nrm{u_n}2=0$ mean that the sequence $\seq un$ concentrates: there exists a sequence $\seq yi$ of points in $\S^d$ (eventually finite) and two sequences of positive numbers $\seq\zeta{i}$ and $(r_{i,n})_{i,n\in\N}$ such that $\lim_{n\to+\infty}r_{i,n}=0$, $\Sigma_{i\in\N}\zeta_i=1$ and $\int_{\S^d\cap B(y_i,r_{i,n})}|u_{i,n}|^q\,d\sigma=\zeta_i+o(1)$, where $u_{i,n}\in\H^1(\S^d)$, $u_{i,n} =u_n$ on $\S^d\cap B(y_i,r_{i,n})$ and $\mbox{supp }u_{i,n}\subset \S^d\cap B(y_i,2\,r_{i,n})$. Here $o(1)$ means that uniformly with respect to $i$, the remainder term converges towards $0$ as $n\to+\infty$. A computation similar to those of the proof of Proposition \ref{prop-upper}, we can \emph{blow up} each function $u_{i,n}$ and prove
\[
(\alpha_n-\alpha_*)^{\vartheta-1}\iS{\(|\nabla u_{i,n}|^2+\alpha_n\,|u_{i,n}|^2\)}\ge\frac{\mathsf K_{q,d}}{\kappa_{q,d}}\,\zeta_i^{2/q}+o(1)\quad\forall\,i\,.
\]
Let us choose an integer $N$ such that $\(\Sigma_{i=1}^N \zeta_i\)^{2/q}>1-\frac{\kappa_{q,d}\,\eta}{2\,\mathsf K_{q,d}}$. Then we find that
\[
(\alpha_n-\alpha_*)^{\vartheta-1}\iS{\(|\nabla u_n|^2+\alpha_n\,|u_n|^2\)}\ge\frac{\mathsf K_{q,d}}{\kappa_{q,d}}\,\Sigma_1^N\zeta_i^{2/q}+o(1)\ge\frac{\mathsf K_{q,d}}{\kappa_{q,d}}\,\(\Sigma_1^N\zeta_i\)^{2/q}+o(1)\ge \frac{\mathsf K_{q,d}}{\kappa_{q,d}}-\frac\eta2+o(1)\,,
\]
a contradiction with \eqref{asumption-33}.
\end{proof}

For details on the behavior of $\mathsf K_{q,d}$ as $q$ varies, see Proposition \ref{Prop:limitsK}. Collecting all results of this section, this completes the proof of Lemma~\ref{Lem:q>2}.

%%%%%%%%%%%%%%%%%%%%%%%%%%%%%%%%%%%%%%%%%%%%%%%%%%%%%%%%%%%%%%%%%%%%%%
%%%%%%%%%%%%%%%%%%%%%%%%%%%%%%%%%%%%%%%%%%%%%%%%%%%%%%%%%%%%%%%%%%%%%%
\section{Spectral estimates for the Schr\"odinger operator on the sphere}\label{Sec:Spectral}

This section is devoted to the proof of Theorem \ref{Thm1}. As a consequence of the results of Lemma~\ref{Lem:q>2}, the function $\alpha\mapsto\mu(\alpha)$ is invertible, of inverse $\mu\mapsto\alpha(\mu)$, if $d=1$, $2$ or $d\ge3$ and $q<2^*$, and we have the inequality
\be{GeneralizedInterp2}
\iS{|\nabla u|^2}-\mu\(\iS{|u|^q}\)^\frac 2q\ge-\,\alpha(\mu)\iS{|u|^2}\quad\forall\,u\in\H^1(\S^d,d\sigma)\,,\quad\forall\,\mu>0\,.
\ee
Moreover, the function $\mu\mapsto\alpha(\mu)$ is monotone increasing, convex, satisfies $\alpha(\mu)=\mu$ for any $\mu\in(0,\frac d{q-2}]$ and $\alpha(\mu)>\mu$ for any $\mu>d/(q-2)$.

Consider the Schr\"odinger operator $-\Delta-V$ for some function $V\in\L^p(\S^d)$ and the corresponding energy functional
\[
\mathcal E[u]\;:=\iS{|\nabla u|^2}-\iS{V\,|u|^2}\,.
\]
Let
\[
\lambda_1(-\Delta-V)\quad:=\inf_{\begin{array}{c}u\in\H^1(\S^d,d\sigma)\\ \iS{|u|^2}=1\end{array}}\mathcal E[u]\,.
\]
By H\"older's inequality, we have
\[
\mathcal E[u]\ge\iS{|\nabla u|^2}-\nrm{V_+}p\,\nrm uq^2\,,
\]
with $\frac 1p+\frac 2q=1$. From Section~\ref{Sec:Interpolationq>2}, with $\mu=\nrm{V_+}p$, we deduce
\[
\mathcal E[u]\ge-\,\alpha(\mu)\,\nrm u2^2\quad\forall\,u\in\H^1(\S^d,d\sigma)\,,\quad\forall\,V\in\L^p(\S^d)\,,
\]
which amounts to a Keller-Lieb-Thirring inequality on the sphere~\eqref{KLT}, or equivalently
\[\label{Master}
\iS{|\nabla u|^2}-\iS{V\,|u|^2}+\alpha\(\nrm{V_+}p\)\iS{|u|^2} \ge 0 \quad\forall\,u\in\H^1(\S^d,d\sigma) \,,\quad\forall\,V\in\L^p(\S^d)\,.
\]
Notice that this inequality contains simultaneously \eqref{KLT} and \eqref{GeneralizedInterp2}, by optimizing either on $u$ or on $V$.

Optimality in \eqref{KLT} still needs to be proved. This can be done by taking an arbitrary $\mu\in(0,\infty)$ and considering an optimal function for \eqref{GeneralizedInterp2}, for which we have
\[
\iS{|\nabla u|^2}-\mu\(\iS{|u|^q}\)^\frac 2q=\alpha(\mu)\iS{|u|^2}\,.
\]
Because the above expression is homogeneous of degree two, there is no restriction to assume that $\iS{|u|^q}=1$ and, since the solution is optimal, it solves the Euler-Lagrange equation
\[
-\,\Delta u-V\,u=\alpha(\mu)\,u
\]
with $V=\mu\,u^{q-2}$, such that
\[
\nrm{V_+}p=\mu\,\nrm uq^{q/p}=\mu\,.
\]
Hence such a function $V$ realizes the equality in \eqref{KLT}.

Taking into account Lemma~\ref{Lem:q>2} and \eqref{Eq:KLT-GNS}, this completes the proof of Theorem \ref{Thm1} in the general case. The case $d=1$ and $\gamma=1/2$ has to be treated specifically. Using $u\equiv1$ as a test function, we know that $|\lambda_1(-\Delta-V)|\le\mu=\int_{\S^1}V\,dx$. On the other hand consider $u\in\H^1(\S^1)$ such that $\|u\|_{\L^2(\S^1)}=1$. Since $\H^1(\S^1)$ is embedded into $C^{0,1/2}(\S^1)$, there exists $x_0\in\S^1\approx[0,2\,\pi)$ such that $u(x_0)=1$ and
\[
\,|u(x)|^2-1\,=2\int_{x_0}^xu(y)\,u'(y)\,dy=2\int_{x_0+2\pi}^xu(y)\,u'(y)\,dy
\]
can be estimated by
\begin{multline*}
\left|\,|u(x)|^2-1\,\right|\le2\int_{x_0}^x|u(y)|\,|u'(y)|\,dy=2\int_{x_0+2\pi}^x|u(y)|\,|u'(y)|\,dy\\
\le\int_0^{2\pi}|u(y)|\,|u'(y)|\,dy\le\(\int_0^{2\pi}|u(y)|^2\,dy\int_0^{2\pi}|u'(y)|^2\,dy\)^{1/2}
\end{multline*}
using the Cauchy-Schwarz inequality, that is
\[
\left|\,|u(x)|^2-1\,\right|\le2\,\pi\,\|u'\|_{\L^2(\S^1)}\,,
\]
since $\|u'\|_{\L^2(\S^1)}^2=\frac 1{2\,\pi}\int_0^{2\pi}|u'(y)|^2\,dy$ and $\|u\|_{\L^2(\S^1)}^2=\frac 1{2\,\pi}\int_0^{2\pi}|u(y)|^2\,dy=1$ (recall that $d\sigma$ is a probability measure). Thus we get
\[
|u(x)|^2\le1+2\,\pi\,\|u'\|_{\L^2(\S^1)}\,,
\]
{}from which it follows that
\[
\lambda_1(-\Delta-V)\ge\|u'\|_{\L^2(\S^1)}^2-\mu\(1+2\,\pi\,\|u'\|_{\L^2(\S^1)}\)\ge-\mu-\pi^2\,\mu^2\,.
\]
This shows that $\mu\le\alpha(\mu)\le\mu+\pi^2\,\mu^2$. By the Arzel\`a-Ascoli theorem, the embedding of $\H^1(\S^1)$ into $C^{0,1/2}(\S^1)$ is compact. When $d=1$ and $\gamma=1/2$, the proof of the asymptotic behavior of $\alpha(\mu)$ as $\mu\to+\infty$ can then be completed as in the other cases.

%%%%%%%%%%%%%%%%%%%%%%%%%%%%%%%%%%%%%%%%%%%%%%%%%%%%%%%%%%%%%%%%%%%%%%
%%%%%%%%%%%%%%%%%%%%%%%%%%%%%%%%%%%%%%%%%%%%%%%%%%%%%%%%%%%%%%%%%%%%%%
\section{Spectral inequalities in the case of positive potentials}\label{Sec:Confining}

In this section we address the case of Schr\"odinger operators $-\Delta +W$ where $W$ is a positive potential on $\S^d$ and we derive estimates from below for the first eigenvalue of such operators. In order to do so, we first study interpolation inequalities in the Euclidean space $\R^d$, like those studied in Section \ref{Sec:Interpolationq>2} (for $q>2$).

For this purpose, let us define for $q\in (0,2)$ the constant
\[
\mathsf K^*_{q,d} :=\inf_{v\in\H^1(\R^d)\setminus\{0\}}\frac{\nrmRd{\nabla v}2^2+\nrmRd vq^2}{\nrmRd v2^2}\,,
\]
that is the the optimal constant in the Gagliardo-Nirenberg-Sobolev inequality
\be{GNS2}
\mathsf K^*_{q,d}\,\nrmRd v2^2\le\nrmRd{\nabla v}2^2+\nrmRd vq^2\quad\forall\,v\in\H^1(\R^d)
\ee
(with the convention that the r.h.s.~is infinite if $|v|^q$ is not integrable).

The optimal constant $\L_{-\gamma,d}^1$ in~\eqref{L-gamma,d} is such that
\be{Eq:KLT-GNS2}
\L_{-\gamma,d}^1:=\( \mathsf K^*_{q,d}\)^{-\gamma}\quad\mbox{with}\quad q=2\,\frac{2\,\gamma-d}{2\,\gamma-d+2}\,.
\ee
See Appendix~\ref{Sec:GNS-KLT2} for a proof. Let us define the exponent
\[
\delta:=\frac{2\,q}{2\,d-q\,(d-2)}\,.
\]
%---------------------------------------------------------------------
\begin{lem}\label{Lem:q<2} Let $q\in(0,2)$ and $d\ge1$. Then there exists a concave increasing function $\nu:\R^+\to\R^+$ with the following properties:
\[
\nu(\beta)\le\beta\quad\forall\,\beta>0\quad\mbox{and}\quad\nu(\beta)<\beta\quad\forall\,\beta\in\big(\tfrac d{2-q},+\infty\big)\,,
\]
\[
\nu(\beta)=\beta\quad\forall\,\beta\in\big[0,\tfrac d{2-q}\big]\quad\mbox{if}\quad q\in [1,2)\,,\quad\mbox{and}\quad\;\lim_{\beta\to 0_+}\frac{\nu(\beta)}{\beta}=1\quad\mbox{if}\quad q\in(0,1)\,,
\]
\[
\nu(\beta)=\mathsf K^*_{q,d}\,\(\kappa_{q,d}\,\beta\)^\delta\,(1+o(1))\quad\mbox{as}\quad\beta\to+\infty\;,
\]
such that
\be{InterpSphere2}
\nrm{\nabla u}2^2+\beta\,\nrm uq^2\ge\nu(\beta)\,\nrm u2^2\quad\forall\,u\in\H^1(\S^d)\,.
\ee
\end{lem}
%---------------------------------------------------------------------
\begin{proof} Inequality~\eqref{InterpSphere2} is obtained by minimizing the l.h.s.~under the constraint $\nrm u2=1$: there is a minimizer which satisfies
\[\label{EL2}
-\,\Delta u+\beta\,u^{q-1}-\nu(\beta)\,u=0\,.
\]

\smallskip\noindent\emph{Case $q\in (1,2)$.} The proof is very similar to that of Lemma \ref{Lem:q>2}, so we leave it to the reader. Written for the optimal value of $\nu(\beta)$, inequality~\eqref{InterpSphere2} is optimal in the following sense:
\begin{enumerate}
\item[(i)] If $0<\beta\le d/(2-q)$, equality is achieved by constants. See \cite{DEKL} for rigidity results on $\S^d$.
\item[(ii)] If $\beta=d/(2-q)$, the sequence $\seq un$ with $u_n:=1+\frac 1n\,\varphi$ where $\varphi$ is an eigenfunction of the Laplace-Beltrami operator, is a minimizing sequence of the quotient to the l.h.s.~of~\eqref{InterpSphere2} divided by the r.h.s.~which converges to the optimal value of $\nu(\beta)=\beta=d/(2-q)$, that is,
\[
\lim_{n\to\infty}\frac{\nrm{\nabla u_n}2^2}{\nrm{u_n}2^2-\nrm{u_n}q^2}=\frac d{2-q}\,.
\]
\item[(iii)] If $\beta>d/(2-q)$, there exists a non-constant positive function $u\in\H^1(\S^d)\setminus\{0\}$ such that equality holds in~\eqref{InterpSphere2}.
\end{enumerate}

\smallskip\noindent\emph{Case $q\in (0,1]$.} In this case, since $\S^d$ is compact, the case $q\le 1$ does not differ from the case $q\in (1,2)$ as far as the existence of $\nu(\beta)$ is concerned. The only difference is that there is no known rigidity result for $q<1$. However we can prove that
\[
\lim_{\beta\to0_+}\frac{\nu(\beta)}\beta=1\,.
\]
Indeed, let us notice that $\nu(\beta)\le\beta$ (use constants as test functions). On the other hand, let $u_\beta=c_\beta+v_\beta$ be a minimizer for $\nu(\beta)$ such that $c_\beta=\iS{u_\beta}$ and, as a consequence, $\iS{v_\beta}=0$. Without l.o.g.~we can set $\iS{|c_\beta+v_\beta|^2}=c_\beta^2+\iS{|v_\beta|^2}=1$. Using the Poincar\'e inequality, we know that $\nrm{\nabla v_\beta}2^2\ge d\,\nrm{v_\beta}2^2$ and hence
\[
d\,\nrm{v_\beta}2^2+\beta\,\nrm{c_\beta+v_\beta}q^2\le\nrm{\nabla v_\beta}2^2+\beta\,\nrm{c_\beta+v_\beta}q^2=\nu(\beta)\le\beta
\]
which shows that $\lim_{\beta\to0_+}\nrm{v_\beta}2=0$ and $\lim_{\beta\to0_+}c_\beta=1$. As a consequence, $\nrm{c_\beta+v_\beta}q^2=c_\beta^2\,(1+o(1))$ as $\beta\to0_+$ and we obtain that
\[
\beta\,(1+o(1))=\beta\,c_\beta^2\,(1+o(1))\le\nu(\beta)\,,
\]
which concludes the proof.

\smallskip\noindent\emph{Asymptotic behavior of $\nu(\beta)$.} Finally, the asymptotic behavior of $\nu(\beta)$ when $\beta$ is large can be investigated using concentration-compactness methods similar to those used in the proofs of Propositions \ref{Cor:CriticalSubCritical}, \ref{prop-upper} and \ref{prop-asympt}. Details are left to the reader.
\end{proof}

\begin{proof}[Proof of Theorem~\ref{Thm2}] By H\"older's inequality we have
\[
\nrm uq^2=\(\iS{W^{-\frac q2}\,\(W\,|u|^2\)^\frac q2}\)^{2/q}\le\nrm{W^{-1}}{\frac q{2-q}}\,\,\iS{W\,|u|^2}\,.
\]
Using \eqref{InterpSphere2}, we get
\[
\iS{|\nabla u|^2}+\iS{W\,|u|^2}\ge\iS{|\nabla u|^2}+\nrm{W^{-1}}p^{-1}\,\nrm uq^2 \ge \nu\(\nrm{W^{-1}}p^{-1}\)\,\iS{|u|^2}
\]
with $p=q/(2-q)$, which proves~\eqref{T31}. Then Theorem~\ref{Thm2} is an easy consequence of Lemma~\ref{Lem:q<2}.
\end{proof}

%%%%%%%%%%%%%%%%%%%%%%%%%%%%%%%%%%%%%%%%%%%%%%%%%%%%%%%%%%%%%%%%%%%%%%
%%%%%%%%%%%%%%%%%%%%%%%%%%%%%%%%%%%%%%%%%%%%%%%%%%%%%%%%%%%%%%%%%%%%%%
\section{The threshold case: \texorpdfstring{$q=2$}{q=2}}\label{Sec:Conclusion}

The limiting case $q=2$ in the interpolation inequality~\eqref{Ineq:Interpolation} corresponds to the logarithmic Sobolev inequality
\[\label{Ineq:LogSob}
\iS{|u|^2\log\(\frac{|u|^2}{\nrm u2^2}\)}\le\frac2d\iS{|\nabla u|^2}\quad\forall\,u\in\H^1(\S^d,d\sigma)
\]
which has been studied, \emph{e.g.,}~in \cite{MR1230930, Brouttelande-03a, Brouttelande-03b}. For earlier results on the sphere, see \cite{federbush1969partially, Rothaus-81, MR674060} and references therein (in particular for the circle). Now, if we consider inequality~\eqref{InterpSphere}, in the limiting case $q=2$ we obtain the following interpolation inequality.
%---------------------------------------------------------------------
\begin{lem}\label{Lem:q=2} For any $p>\max\{1,d/2\}$, there exists a concave nondecreasing function $\xi:(0\,,+\infty)\to\R$ with the properties
\[
\xi(\alpha)=\alpha\quad\forall\;\alpha\in(0,\alpha_0)\quad\mbox{and}\quad\xi(\alpha)<\alpha\quad\forall\,\alpha>\alpha_0
\]
for some $\alpha_0\in\big[\frac d2\,(p-1),\frac d2\,p\big]$, and
\[
\xi(\alpha)\sim\alpha^{1-\frac d{2\,p}}\quad\mbox{as}\quad\alpha\to+\infty
\]
such that
\be{InterpLogSob}
\iS{|u|^2\log\(\frac{|u|^2}{\nrm u2^2}\)}+p\,\log\big(\tfrac{\xi(\alpha)}\alpha\big)\,\nrm u2^2\le p\,\nrm u2^2\,\log\(1+\frac{\nrm{\nabla u}2^2}{\alpha\,\nrm{ u}2^2}\)\quad\forall\,u\in\H^1(\S^d)\,.
\ee
\end{lem}
%---------------------------------------------------------------------
\begin{proof} Consider H\"older's inequality: $\nrm ur\le\nrm u2^\theta\,\nrm uq^{1-\theta}$, with $2\le r<q$ and $\theta=\frac 2r\,\frac{q-r}{q-2}$. To emphasize the dependence of $\theta$ in $r$, we shall write $\theta=\theta(r)$. By taking the logarithm of both sides of the inequality, we find that
\[
\frac 1r\log \iS{|u|^r} \le\frac{\theta(r)}2\,\log \iS{|u|^2} +\frac{1-\theta(r)}q\,\log\iS{|u|^q}\,.
\]
The inequality becomes an equality when $r=2$, so that we may differentiate at $r=2$ and get, with $q=\frac{2\,p}{p-1}<2^*$, \emph{i.e.}~$p=\frac q{q-2}$, the \emph{logarithmic H\"older} inequality
\[\label{Ineq:logHoelder}
\iS{|u|^2\log\(\frac{|u|^2}{\nrm u2^2}\)}\le p\,\nrm u2^2\,\log\(\frac{\nrm uq^2}{\nrm u2^2}\)\quad\forall\,u\in\H^1(\S^d)\,.
\]
We may now use inequality~\eqref{InterpSphere} to estimate
\[
\frac{\nrm uq^2}{\nrm u2^2}\le\frac\alpha{\mu(\alpha)}\,\(1+\frac 1\alpha\,\frac{\nrm{\nabla u}2^2}{\nrm u2^2}\)
\]
where $\mu=\mu(\alpha)$ is the constant which appears in Lemma~\ref{Lem:q>2}. Thus we get
\[\label{Ineq:logHoelderbis}
\iS{|u|^2\log\(\frac{|u|^2}{\nrm u2^2}\)}+p\,\log\big(\tfrac{\mu(\alpha)}\alpha\big)\,\nrm u2^2\,\le p\,\nrm u2^2\,\log\(1+\frac{\nrm{\nabla u}2^2}{\alpha\,\nrm u2^2}\)
\,,
\]
which proves that the inequality
\[
\iS{|u|^2\log\(\frac{|u|^2}{\nrm u2^2}\)}+p\,\log\xi(\alpha)\,\nrm u2^2\le p\,\nrm u2^2\,\log\(\alpha+\frac{\nrm{\nabla u}2^2}{\nrm{ u}2^2}\)
\]
holds for some optimal constant $\xi(\alpha)\ge\mu(\alpha)$, which is therefore concave and such that $\lim_{\alpha\to+\infty}\xi(\alpha)=+\infty$. This establishes \eqref{InterpLogSob}. The fact that equality is achieved for every $\alpha>0$ follows from the method of \cite[Proposition 3.3]{1005}.

Testing \eqref{InterpLogSob} with constant functions, we find that $\xi(\alpha)\le\alpha$ for any $\alpha>0$. On the other hand, $\xi(\alpha)\ge\mu(\alpha)=\alpha$ for any $\alpha\le\frac d{q-2}=\frac d2\,(p-1)$. Testing \eqref{InterpLogSob} with $u=1+\eps\,\varphi$, we find that $\xi(\alpha)<\alpha$ if $\alpha>\frac d2\,p$.

By Proposition \ref{prop-asympt}, we know that $\xi(\alpha)\ge\mu(\alpha)\sim\alpha^{1-\vartheta}$ with $\vartheta=d\,\frac{q-2}{2\,q}=\frac d{2\,p}$ as $\alpha\to+\infty$. As in the proof of Propositions \ref{prop-upper} and \ref{prop-asympt}, let us consider an optimal function $u_\alpha$ for \eqref{InterpLogSob}. Then we have
\[
p\,\log\big(\tfrac{\xi(\alpha)}\alpha\big)=p\,\log\(1+\frac1\alpha\,\nrm{\nabla u_\alpha}2^2\)-\iS{|u_\alpha|^2\log|u_\alpha|^2}\sim\frac p\alpha\,\nrm{\nabla u_\alpha}2^2-\iS{|u_\alpha|^2\log|u_\alpha|^2}
\]
as $\alpha\to+\infty$ and $u_\alpha$ concentrates at a single point like in the case $q>2$ so that, after a stereographic projection which transforms $u_\alpha$ into $v_\alpha$, the function $v_\alpha$ is, up to higher order terms, optimal for the Euclidean logarithmic Sobolev inequality
\[
\iRd{|v|^2\log\(\frac{|v|^2}{\nrmRd v2^2}\)}+\frac d2\,\log(\pi\,\eps\,e^2)\,\nrmRd v2^2\le\eps\,\nrmRd{\nabla v}2^2
\]
which holds for any $\eps>0$ and any $v\in\H^1(\R^d)$. Here we have of course $\eps=p/\alpha$ and find that
\[
p\,\log\big(\tfrac{\xi(\alpha)}\alpha\big)=\tfrac d2\,\log\big(\pi\,\tfrac p\alpha\,e^2\big)\,(1+o(1))\quad\mbox{as}\quad\alpha\to+\infty\,,
\]
which concludes the proof.
\end{proof}

%---------------------------------------------------------------------
\begin{cor}\label{Cor:LinearizedLogSob} With the notations of Lemma~\ref{Lem:q=2}, for any $\alpha>0$ we have
\[
\frac\alpha p\,\iS{|u|^2\log\(\frac{|u|^2}{\nrm u2^2}\)}+\alpha\,\log\big(\tfrac{\xi(\alpha)}\alpha\big)\,\nrm u2^2\le\nrm{\nabla u}2^2\quad\forall\,u\in\H^1(\S^d)\,.
\]
\end{cor}
%---------------------------------------------------------------------
\begin{proof} This is a straightforward consequence of Lemma~\ref{Lem:q=2} using the fact that $\log(1+x)\le x$ for any $x>0$.\end{proof}

As in the case $q\neq2$, Corollary~\ref{Cor:LinearizedLogSob} provides some spectral estimates. Let $u\in\H^1(\S^d)$ be such that $\nrm u2=1$. A straightforward optimization with respect to an arbitrary function $W$ shows that
\[
\inf_W\left[\iS{W\,|u|^2}+\mu\,\log\(\iS{e^{-W/\mu}}\)\right]=-\,\mu\iS{|u|^2\,\log|u|^2}\,,
\]
with optimality case achieved by $W$ such that
\[
|u|^2=\frac{e^{-W/\mu}}{\iS{e^{-W/\mu}}}\,.
\]
Notice that, up to the addition of a constant, we can always assume that $\iS{e^{-W/\mu}}=1$, which uniquely determines the optimal~$W$. Now, by Corollary~\ref{Cor:LinearizedLogSob} applied with $\mu=\alpha/p$, we find that
\[
\iS{|\nabla u|^2}+\iS{W\,|u|^2}\ge\alpha\,\log\big(\tfrac{\xi(\alpha)}\alpha\big)-\frac\alpha p\,\log\(\iS{e^{-\,p\,W/\alpha}}\)\,.
\]
This leads us to the following statement.
%---------------------------------------------------------------------
\begin{cor}\label{Cor:q=2} Let $d\ge1$. With the notations of Lemma~\ref{Lem:q=2}, we have the following estimate
\[
e^{-\,\lambda_1(-\Delta-W)/\alpha}\le\frac\alpha{\xi(\alpha)}\,\(\iS{e^{-\,p\,W/\alpha}}\)^{1/p}
\]
for any function $W$ such that $e^{-\,p\,W/\alpha}$ is integrable. This estimate is optimal in the sense that there exists a nonnegative function $W$ for which the inequality becomes an equality.\end{cor}
%---------------------------------------------------------------------

%%%%%%%%%%%%%%%%%%%%%%%%%%%%%%%%%%%%%%%%%%%%%%%%%%%%%%%%%%%%%%%%%%%%%%
%%%%%%%%%%%%%%%%%%%%%%%%%%%%%%%%%%%%%%%%%%%%%%%%%%%%%%%%%%%%%%%%%%%%%%
\appendix\section{Further estimates and numerical results}\label{Sec:AppendixA}

%%%%%%%%%%%%%%%%%%%%%%%%%%%%%%%%%%%%%%%%%%%%%%%%%%%%%%%%%%%%%%%%%%%%%%
\subsection{A refined upper estimate}\label{Sec:RefinedUpper}

Let $q\in(2,2^*)$. For $\alpha>d/(q-2)$, we can give an upper estimate of the optimal constant $\mu(\alpha)$ in inequality~\eqref{InterpSphere} of Lemma~\ref{Lem:q>2}. Consider functions which depend only on $z$, with the notations of Section~\ref{Sec:PropertiesAlpha(mu)}. Then \eqref{InterpSphere} is equivalent to an inequality that can be written as
\[
\mathsf F_\alpha[f]:=\frac{\int_{-1}^1|f'|^2\;\nu\;d\nu_d+\alpha\int_{-1}^1|f|^2\;d\nu_d}{\(\int_{-1}^1|f|^q\;d\nu_d\)^{2/q}}\ge\mu(\alpha)
\]
where $d\nu_d$ is the probability measure defined by
\[
\nu_d(z)\,dz=d\nu_d(z):=Z_d^{-1}\,\nu^{\frac d2-1}\,dz\quad\mbox{with}\quad\nu(z):=1-z^2\;,\quad Z_d:=\sqrt\pi\,\frac{\Gamma(\tfrac d2)}{\Gamma(\tfrac{d+1}2)}\,.
\]
See \cite{DEKL} for details. To get an estimate, it is enough to take a well chosen test function: consider $f_\eps(z):=1+\eps\,\varphi(z)$ and as in Section~\ref{Sec:PropertiesAlpha(mu)} we can choose $\varphi(z)=z$. Then one can optimize $h_\alpha(\eps)=\mathsf F_\alpha[f_\eps]$ with respect to $\eps\in(0,1)$, and observe that $\int_{-1}^1|f_\eps'|^2\;\nu\;d\nu_d=d\,\eps^2\int_{-1}^1z^2\;d\nu_d$, so that
$h_\alpha(\eps)$ can be written as
\[
h_\alpha(\eps)=\frac{\alpha+(d+\alpha)\,\eps^2\int_{-1}^1z^2\;d\nu_d}{\(\int_{-1}^1|1+\eps\,z|^q\;d\nu_d\)^{2/q}}\ge\mu(\alpha)\,.
\]
When $\eps\to0_+$, we recover that $h_\alpha(\eps)-\alpha\sim\big[d-\alpha\,(q-2)\big]\,\eps^2\int_{-1}^1z^2\;d\nu_d<0$ if $\alpha>d/(q-2)$, but a better estimate can be achieved simply by considering $\mu_+(\alpha):=\inf_{\eps\in(0,1)}h_\alpha(\eps)$ so that $\mu(\alpha)\le\mu_+(\alpha)<\alpha$. The function $\alpha\mapsto\mu_+(\alpha)$ can be computed explicitly (using hypergeometric functions) and is shown in Fig.~1.

%%%%%%%%%%%%%%%%%%%%%%%%%%%%%%%%%%%%%%%%%%%%%%%%%%%%%%%%%%%%%%%%%%%%%%
\subsection{Numerical results}

In this section, we illustrate the various estimates obtained in this paper by numerical computations done in the special case $d=3$ and $q=3$. See Fig.~1 for the computation of the curve $\alpha\mapsto\mu(\alpha)$ and how it behaves compared to the theoretical estimates obtained in this paper. We emphasize that our upper and lower estimates $\alpha\mapsto\mu_\pm(\alpha)$ bifurcate from the line $\mu=\alpha$ precisely at $\alpha=d/(q-2)$ if $q\in(2,2^*)$ (and at $\alpha=d/(2-q)$ if $q\in(1,2)$). The curve corresponding to the asymptotic regime is also plotted, but gives relevant information only as $\alpha\to\infty$.
%---------------------------------------------------------------------
\begin{figure}[hb]
\includegraphics[width=9cm]{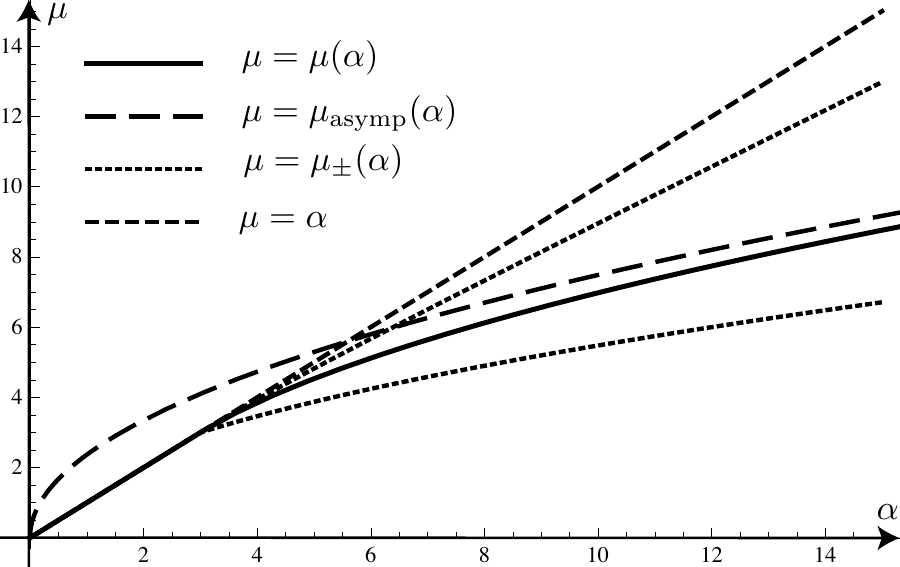}
\caption{\it In the case $q>2$, the optimal constant is given by $\mu=\alpha$ for $\alpha\le d/(q-2)$ and the curve $\mu=\mu(\alpha)$ for $\alpha>d/(q-2)$. An upper estimate is given by the curve $\mu=\mu_+(\alpha)$ obtained by optimizing the function $h_\alpha(\eps)$ in terms of $\eps\in(0,1)$ while a lower estimate, namely $\mu=\mu_-(\alpha)=\alpha_*^\vartheta\,\alpha^{1-\vartheta}$ has been established in Proposition~\ref{Cor:CriticalSubCritical}. The asymptotic regime is governed by $\mu(\alpha)\sim\mu_{\rm asymp}(\alpha)=\mathsf K_{q,d}\,\kappa_{q,d}^{-1}\,\alpha^{1-\vartheta}$ as $\alpha\to+\infty$ according to Lemma~\ref{Lem:q>2}. The above plot shows the various curves in the special case $d=3$ and $q=3$.}
\end{figure}
%---------------------------------------------------------------------

The convergence towards the asymptotic regime is illustrated in Fig.~2 which shows the convergence of $\mu(\alpha)/\mu_{\rm asymp}(\alpha)$ towards $1$ as $\alpha\to+\infty$ in the special case $d=3$ and $q=3$. In terms of spectral properties, for large potentials, eigenvalues of the Schr\"odinger operator can be estimated according to Theorem~\ref{Thm1} by the Euclidean Keller-Lieb-Thirring constant that has been numerically computed for instance in~\cite[Appendix~A. Numerical studies, by J.F.~Barnes]{Lieb-Thirring76}.

%---------------------------------------------------------------------
\begin{figure}[ht]
\includegraphics[width=9cm]{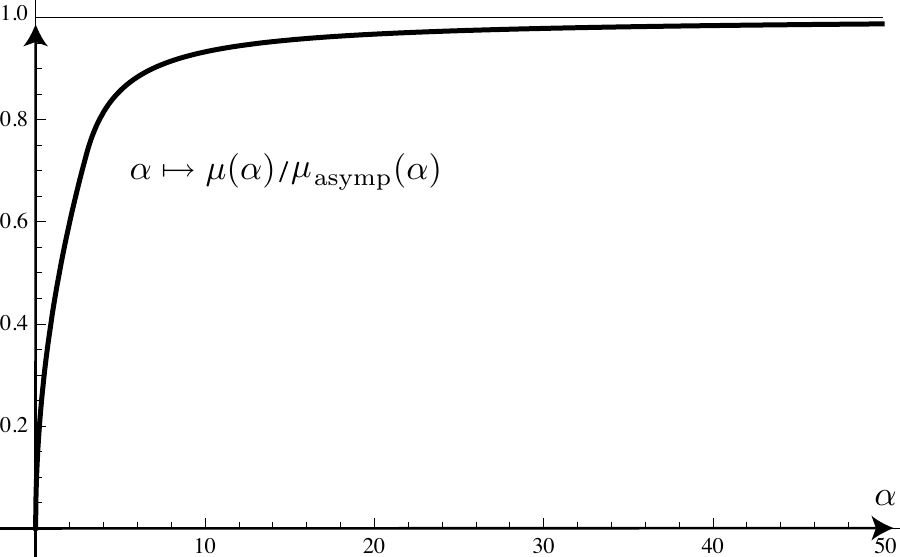}
\caption{\it The asymptotic regime corresponding to $\alpha\to+\infty$ has the interesting feature that, up to a dependence in $\alpha^{1-\vartheta}$ and a normalization factor proportional to $\kappa_{q,d}$, the optimal constant $\mu(\alpha)$ behaves like the optimal constant in the Euclidean space, as has been established in Proposition~\ref{prop-asympt}.}
\end{figure}
%---------------------------------------------------------------------

%%%%%%%%%%%%%%%%%%%%%%%%%%%%%%%%%%%%%%%%%%%%%%%%%%%%%%%%%%%%%%%%%%%%%%
%%%%%%%%%%%%%%%%%%%%%%%%%%%%%%%%%%%%%%%%%%%%%%%%%%%%%%%%%%%%%%%%%%%%%%
\newpage\section{Constants on the Euclidean space}\label{Appendix}

%%%%%%%%%%%%%%%%%%%%%%%%%%%%%%%%%%%%%%%%%%%%%%%%%%%%%%%%%%%%%%%%%%%%%%
\subsection{Scaling of the Gagliardo-Nirenberg-Sobolev inequality}\label{Sec:GNS}

Let $q>2$ and denote by $\mathsf K_{\rm GN}(q)$ the optimal constant in the Gagliardo-Nirenberg-Sobolev inequality, given by
\[
\mathsf K_{\rm GN}(q):=\inf_{u\in\H^1(\R^d)\setminus\{0\}}\frac{\nrmRd{\nabla u}2^{2\,\vartheta}\,\nrmRd u2^{2\,(1-\,\vartheta)}}{\nrmRd uq^2}\quad\mbox{with}\quad\vartheta=\vartheta(q,d)=d\,\frac{q-2}{2\,q}\,.
\]
An optimization of the quotient in the definition of $\mathsf K_{q,d}$, which has been defined in Section~\ref{Sec:Interpolationq>2}, allows to relate this constant with $\mathsf K_{\rm GN}(q)$. Indeed, if we optimize $\mathcal N[u]:=\iRd{|\nabla u|^2}+\iRd{|u|^2}$ under the scaling $\lambda\mapsto u_\lambda(x):=\lambda^{d/q}\,u(\lambda\,x)$, then we find that
\[
\mathcal N[u_\lambda]=\lambda^{2\,(1-\,\vartheta)}\iRd{|\nabla u|^2}+\lambda^{-\,2\,\vartheta}\iRd{|u|^2}
\]
achieves its minimum at
\[
\lambda_\star=\sqrt{\frac\vartheta{1-\,\vartheta}}\,\frac{\nrmRd u2}{\nrmRd{\nabla u}2}\,,
\]
so that
\[
\mathcal N[u_{\lambda_\star}]=\vartheta^{-\,\vartheta}\,(1-\,\vartheta)^{-\,(1-\,\vartheta)}\,\nrmRd{\nabla u}2^{2\,\vartheta}\,\nrmRd u2^{2\,(1-\,\vartheta)}\,,
\]
thus proving that $\mathsf K_{q,d}$ can be computed in terms of $\mathsf K_{\rm GN}(q)$ as
\[
\mathsf K_{q,d}=\vartheta^{-\,\vartheta}\,(1-\,\vartheta)^{-\,(1-\,\vartheta)}\,\mathsf K_{\rm GN}(q)\,.
\]

%%%%%%%%%%%%%%%%%%%%%%%%%%%%%%%%%%%%%%%%%%%%%%%%%%%%%%%%%%%%%%%%%%%%%%
\subsection{Asymptotic regimes in Gagliardo-Nirenberg-Sobolev inequalities}\label{Sec:AsympGNS}

Let $q>2$ and consider the constant $\mathsf K_{q,d}$ as above. To handle the case of dimension $d=1$, we may observe that for any smooth compactly supported function $u$ on $\R$, we can write either
\[
|u(x)|^2=2\,\left|\,\int_{-\infty}^xu(y)\,u'(y)\;dy\;\right|\le\|u\|_{\L^2(-\infty,x)}^2+\|u'\|_{\L^2(-\infty,x)}^2\quad\forall\,x\in\R
\]
or
\[
|u(x)|^2=2\,\left|\,\int_x^{+\infty}u(y)\,u'(y)\;dy\;\right|\le\|u\|_{\L^2(x,+\infty)}^2+\|u'\|_{\L^2(x,+\infty)}^2\quad\forall\,x\in\R
\]
thus proving that
\[
|u(x)|^2\le\frac 12\(\|u\|_{\L^2(\R)}^2+\|u'\|_{\L^2(\R)}^2\)\quad\forall\,x\in\R\,,
\]
that is, the Agmon inequality
\[
\frac{\|u\|_{\L^2(\R)}^2+\|u'\|_{\L^2(\R)}^2}{\|u\|_{\L^\infty(\R)}^2}\ge2\,,
\]
and hence $\mathsf K_{\infty,1}\ge2$. Equality is achieved by the function $u(x)=e^{-|x|}$, $x\in\R$, and we have shown that
\[
\mathsf K_{\infty,1}=2\,.
\]
%---------------------------------------------------------------------
\begin{prop}\label{Prop:limitsK}
Assume that $q>2$. For all $d\geq 1$,
\[
\lim_{q\to 2_+}\mathsf K_{q,d}=1
\]
and, for all $d\geq 3$,
\[
\lim_{q\to2^*}\mathsf K_{q,d}=\mathsf S_d
\]
where $\mathsf S_d$ is the best constant in inequality~\eqref{Ineq:Sobolev}. If $d=1$, then $\lim_{q\to +\infty}\mathsf K_{q,1}=\mathsf K_{\infty,1}$.\end{prop}
%---------------------------------------------------------------------
\begin{proof} For any $v\in\H^1(\R^d)$ and $d\ge 3$, we have
\[
\lim_{q\to2^*}\frac{\nrmRd{\nabla v}2^2+\nrmRd v2^2}{\nrmRd vq^2}\ge\lim_{q\to2^*}\frac{\nrmRd{\nabla v}2^2}{\nrmRd vq^2}=\frac{\nrmRd{\nabla v}2^2}{\nrmRd v{2^*}^2}\ge\mathsf S_d\,,
\]
thus proving that $\lim_{q\to2^*}\mathsf K_{q,d}\ge\mathsf S_d$. On the other hand, we may use the Aubin-Talenti function
\be{Eqn:Optimal}
\overline u(x)=(1+|x|^2)^{-\frac{d-2}2}\quad\forall\,x\in\R^d
\ee
as test function for $\mathsf K_{q,d}$ if $d\ge5$, \emph{i.e.}
\[
\mathsf K_{q,d}\le\vartheta^{-\,\vartheta}\,(1-\,\vartheta)^{-\,(1-\,\vartheta)}\,\frac{\nrmRd{\nabla\overline u}2^{2\,\vartheta}\,\nrmRd{\overline u}2^{2\,(1-\,\vartheta)}}{\nrmRd{\overline u}q^2}
\]
and observe that the right-hand side converges to $\mathsf S_d$ since $\lim_{q\to2^*}\vartheta(q,d)=1$. If $d=3$ or $4$, standard additionnal truncations are needed. The case corresponding to $q\to\infty$, $d=1$ is dealt with as above.

Now we investigate the limit as $q\to 2_+$. For any $v\in\H^1(\R^d)$, we have
\[
\lim_{q\to2_+}\frac{\nrmRd{\nabla v}2^2+\nrmRd v2^2}{\nrmRd vq^2}\ge\lim_{q\to2_+}\frac{\nrmRd v2^2}{\nrmRd vq^2}=1\,,
\]
thus proving that $\lim_{q\to2_+}\mathsf K_{q,d}\ge1$, and for any $v\in\H^1(\R^d)$, the right-hand side in
\[
\mathsf K_{q,d}\le\vartheta^{-\,\vartheta}\,(1-\,\vartheta)^{-\,(1-\,\vartheta)}\,\frac{\nrmRd{\nabla v}2^{2\,\vartheta}\,\nrmRd v2^{2\,(1-\,\vartheta)}}{\nrmRd vq^2}
\]
converges to $1$ as $q\to2_+$. This completes the proof.
\end{proof}

%%%%%%%%%%%%%%%%%%%%%%%%%%%%%%%%%%%%%%%%%%%%%%%%%%%%%%%%%%%%%%%%%%%%%%
\subsection{Stereographic projection}\label{Sec:Stereographic}

On $\S^d\subset\R^{d+1}$, we can introduce the coordinates $y=(\rho\,\phi,z)\in\R^d\times\R$ such that $\rho^2+z^2=1$, $z\in[-1,1]$, $\rho\ge0$ and $\phi\in\S^{d-1}$, and consider the stereographic projection $\Sigma:\S^d\setminus\{\mathrm N\}\to\R^d$ defined by $\Sigma(y)=x$ where, using the above notations, $x=r\,\phi$ with $r=\sqrt{(1+z)/(1-z)}$ for any $z\in[-1,1)$. In this setting the \emph{North Pole} $\mathrm N$ corresponds to $z=1$ (and is formally sent at infinity) while the \emph{equator} (corresponding to $z=0$) is sent onto the unit sphere $\S^{d-1}\subset\R^d$. Hence $x\in\R^d$ is such that $r=|x|$, $\phi=\frac x{|x|}$, and we have the useful formulae
\[
z=\frac {r^2-1}{r^2+1}=1-\frac 2{r^2+1}\;,\quad \rho=\frac {2\,r}{r^2+1}\,.
\]
With these notations in hand, we can transform any function $u$ on $\S^d$ into a function $v$ on $\R^d$ using
\[
u(y)=\big(\tfrac r\rho\big)^\frac{d-2}2\,v(x)=\big(\tfrac{r^2+1}2\big)^\frac{d-2}2\,v(x)=(1-z)^{-\frac{d-2}2}\,v(x)
\]
and a painful but straightforward computation shows that, with $\alpha_*=\frac 14\,d\,(d-2)$,
\[
\int_{\S^d}|\nabla u|^2\;d\omega+\alpha_*\int_{\S^d}|u|^2\;d\omega=\iRd{|\nabla v|^2}\quad\mbox{and}\quad\int_{\S^d}|u|^q\;d\omega=\iRd{|v|^q\,\big(\tfrac2{1+|x|^2}\big)^{d-(d-2)\frac q2}}\,.
\]
As a consequence, Inequalities~\eqref{InterpSphere} and~\eqref{InterpSphere2} are transformed respectively into
\[
\iRd{|\nabla v|^2}+4\,(\alpha-\alpha_*)\int_{\R^d}|v|^2\,\frac{dx}{(1+|x|^2)^2}\ge\mu(\alpha)\,\kappa_{q,d}\left[\,\iRd{|v|^q\,\Big(\tfrac2{1+|x|^2}\Big)^{d-(d-2)\frac q2}}\right]^\frac 2q\quad\forall\,v\in\mathcal D^{1,2}(\R^d)
\]
if $q\in(2,2^*)$ and $\alpha\ge\alpha_*$, and
\[
\iRd{|\nabla v|^2}+\beta\,\kappa_{q,d}\left[\,\iRd{|v|^q\,\Big(\tfrac2{1+|x|^2}\Big)^{d-(d-2)\frac q2}}\right]^\frac 2q\ge 4\,(\nu(\beta)+\alpha_*)\int_{\R^d}|v|^2\,\frac{dx}{(1+|x|^2)^2}\quad\forall\,v\in\mathcal D^{1,2}(\R^d)
\]
if $q\in(1,2)$ and $\beta>0$.

%%%%%%%%%%%%%%%%%%%%%%%%%%%%%%%%%%%%%%%%%%%%%%%%%%%%%%%%%%%%%%%%%%%%%%
\subsection{Sobolev's inequality: expression of the constant and references}\label{Sec:Sobolev}

The proof that Sobolev's inequality \eqref{Ineq:Sobolev} becomes an equality if and only if $u=\overline u$ given by~\eqref{Eqn:Optimal} up to a multiplication by a constant, a translation and a scaling is due to T.~Aubin and G.~Talenti: see \cite{MR0448404,MR0463908}. However, G.~Rosen in \cite{MR0289739} showed (by linearization) that the function given by \eqref{Eqn:Optimal} is a local minimum when $d=3$ and computed the critical value.

Much earlier, G.~Bliss in \cite{bliss1930integral} (also see \cite{Hardy01011930}) established that, among radial functions, the following inequality holds
\[
\(\iRd{|f|^p\,|x|^{r+1-d-p}}\)^\frac 2p\le\mathsf C_{\rm Bliss}\iRd{|\nabla f|^2\,|x|^{1-d}}
\]
when $r=\frac p2-1$. With the change of variables $f(x)=v\(|x|^{-\frac1{d-2}}\,\tfrac x{|x|}\)$, the inequality is changed into
\[
\(\iRd{|v|^\frac{2d}{d-2}}\)^\frac{d-2}d\le\frac{\mathsf C_{\rm Bliss}}{(d-2)^{2\frac{d-1}d}}\iRd{|\nabla v|^2}
\]
if $p=2^*$ and it is a straightforward consequence of \cite{bliss1930integral} that the equality is achieved with $v=\overline u$.

\bigskip According to the duplication formula (see for instance \cite{MR0167642}) for the $\Gamma$ function, we know that
\[
\Gamma(x)\,\Gamma\(x+\tfrac 12\)=2^{1-\,2\,x}\,\sqrt\pi\,\Gamma(2\,x)\,.
\]
As a consequence, the best constant in Sobolev's inequality \eqref{Ineq:Sobolev} can be written either as
\[
\mathsf S_d=\frac 4{d\,(d-2)\,|\mathbb \S^d|^{2/d}}
\]
where the surface of the $d$-dimensional unit sphere is given by $|\mathbb \S^d|=2\,\pi^\frac{d+1}2/\Gamma\(\frac{d+1}2\)$ (see for instance \cite{MR1230930}), or as
\[
\mathsf S_d=\frac 1{\pi\,d\,(d-2)}\,\Big(\tfrac{\Gamma(d)}{\Gamma\(\frac d2\)}\Big)^\frac 2d
\]
according to \cite{MR0448404,bliss1930integral,MR0289739,MR0463908}. This last expression can easily be recovered using the fact that optimality in~\eqref{Ineq:Sobolev} is achieved by $\overline u$ defined in \eqref{Eqn:Optimal}, while the first one, namely $1/\mathsf S_d=\frac14\,d\,(d-2)\,\kappa_{2^*,d}$, is an easy consequence of the stereographic projection and the computations of Section~\ref{Sec:Stereographic} with $\alpha=\alpha_*$ and $q=2^*$.

%%%%%%%%%%%%%%%%%%%%%%%%%%%%%%%%%%%%%%%%%%%%%%%%%%%%%%%%%%%%%%%%%%%%%%
\subsection{A proof of \texorpdfstring{\eqref{Eq:KLT-GNS}}{Eq:KLT-GNS}}\label{Sec:GNS-KLT}

Assume that $q>2$ and let us relate the optimal constant $\L_{\gamma,d}^1$ in the one bound state Keller-Lieb-Thirring inequality~\eqref{Lgamma,d} with the optimal constant $\mathsf K_{q,d}$ in the Gagliardo-Nirenberg-Sobolev inequality~\eqref{Ineq:GNS}. In this case, recall that $p=\frac q{q-2}=\gamma+\frac d2$. For any nonnegative function $\phi$ defined on $\R^d$ such that $\nrmRd\phi p=\mathsf K_{q,d}$, using H\"older's inequality we can write that
\[
\iRd{\(|\nabla v|^2-\phi\,|v|^2\)}\ge\nrmRd{\nabla v}2^2-\nrmRd\phi p\,\nrmRd vq^2
\]
for any $v\in\H^1(\R^d)$. Using~\eqref{Ineq:GNS}, namely
\[
\nrmRd{\nabla v}2^2-\mathsf K_{q,d}\,\nrmRd vq^2\ge-\,\nrmRd v2^2\,,
\]
this proves that
\be{Ineq:Normalized}
|\lambda_1(-\Delta-\phi)|\le1\quad\forall\,\phi\in\L^p(\R^d)\quad\mbox{such that}\quad\nrmRd\phi p=\mathsf K_{q,d}\,.
\ee
Next one can observe that inequality~\eqref{Lgamma,d} can be rephrased as
\[
\L_{\gamma,d}^1=\sup_{\phi\in\L^p(\S^d)}\sup_{v\in\H^1(\R^d)\setminus\{0\}}\big(\mathcal R[v,\phi]\big)^\gamma\quad\mbox{with}\quad\mathcal R[v,\phi]:=\frac{\iRd{\(\phi\,|v|^2-|\nabla v|^2\)}}{\nrmRd v2^2\,\nrmRd\phi p^\frac{2\,p}{2\,p-d}}
\]
where $p=\gamma+d/2$
so that the exponent $\frac{2\,p}{2\,p-d}$ is precisely the one for which we get the scaling invariance of~$\mathcal R$. Indeed, with $v_\lambda(x):=v(\lambda\,x)$ and $\phi_\lambda(x):=\phi(\lambda\,x)$, we get that $\mathcal R[v_\lambda,\lambda^2\,\phi_\lambda]=\mathcal R[v,\phi]$ for any $\lambda>0$. Hence we find that
\[
\sup_{v\in\H^1(\R^d)\setminus\{0\}}\mathcal R[v,\phi]=\frac{|\lambda_1(-\Delta-\phi)|}{\nrmRd\phi p^\frac{2\,p}{2\,p-d}}=\sup_{v\in\H^1(\R^d)\setminus\{0\}}\mathcal R[v_\lambda,\lambda^2\,\phi_\lambda]=\frac{|\lambda_1(-\Delta-\lambda^2\,\phi_\lambda)|}{\nrmRd{\lambda^2\,\phi_\lambda}p^\frac{2\,p}{2\,p-d}}
\]
and if we choose $\lambda$ such that
\[
\lambda^\frac{2\,p-d}p\,\nrmRd\phi p=\nrmRd{\lambda^2\,\phi_\lambda}p=\mathsf K_{q,d}\,,
\]
we obtain
\[
\frac{|\lambda_1(-\Delta-\phi)|}{\nrmRd\phi p^\frac{2\,p}{2\,p-d}}\le\frac 1{\mathsf K_{q,d}^\frac{2\,p}{2\,p-d}}
\]
using~\eqref{Ineq:Normalized}, which proves that $\L_{\gamma,d}^1\le\(\mathsf K_{q,d}\)^{-\,p}$ with $p=\gamma+\frac d2$. Since optimality can be preserved at each step, this actually proves~\eqref{Eq:KLT-GNS}.

See \cite{MR0121101,Lieb-Thirring76,MR1923362,MR2048612,Benguria-Loss03,MR2253013} for further details. In the Euclidean case, notice that the equivalence can be extended to the case of systems on the one hand and to Lieb-Thirring inequalities on the other hand: see \cite{Lieb-Thirring76,Lieb84,MR2253013}.

%%%%%%%%%%%%%%%%%%%%%%%%%%%%%%%%%%%%%%%%%%%%%%%%%%%%%%%%%%%%%%%%%%%%%%
\subsection{A proof of \texorpdfstring{\eqref{Eq:KLT-GNS2}}{Eq:KLT-GNS2}}\label{Sec:GNS-KLT2}

As in \cite{MR2253013}, we can also relate $\L_{-\gamma,d}^1$ and $ \mathsf K^*_{q,d}$ when $q=2\,\frac{2\,\gamma-d}{2\,\gamma-d+2}$ takes values in $(0,2)$. The method is similar to that of Appendix~\ref{Sec:GNS-KLT}. For any function $v\in\H^1(\R^d)$ such that $v^q$ is integrable and any positive potential $\phi$ such that $\phi^{-1}$ is in $\L^p(\R^d)$ with $p=q/(2-q)$, we can use H\"older's inequality as in the proof of Theorem~\ref{Thm2} and get
\[
\iRd{\(|\nabla v|^2+\phi\,|v|^2\)}\ge\nrmRd{\nabla v}2^2+\frac{\nrmRd vq^2}{\nrmRd{\phi^{-1}}p}\,.
\]
Using~\eqref{GNS2}, namely $\nrmRd{\nabla v}2^2+\nrmRd vq^2\ge\mathsf K^*_{q,d}\,\nrmRd v2^2$, this proves that
\[\label{Ineq:Normalized2}
\lambda_1(-\Delta+\phi)\ge \mathsf K^*_{q,d}\quad\forall\,\phi\in\L^p(\R^d)\quad\mbox{such that}\quad\nrmRd{\phi^{-1}}p=1\,.
\]
Inequality~\eqref{L-gamma,d} can be rephrased as
\[
\L_{-\,\gamma,d}^1=\sup_{\phi\in\L^p(\S^d)}\sup_{v\in\H^1(\R^d)\setminus\{0\}}\(\mathcal R[v,\phi]\)^{-\gamma}\quad\mbox{with}\quad\mathcal R[v,\phi]:=\frac{\iRd{\(|\nabla v|^2+\phi\,|v|^2\)}}{\nrmRd v2^2}\,\nrmRd{\phi^{-1}}p^{p/\gamma}
\]
with $\gamma=p+\frac d2$. The same scaling as in Appendix~\ref{Sec:GNS-KLT} applies: with $v_\lambda(x):=v(\lambda\,x)$ and $\phi_\lambda(x):=\phi(\lambda\,x)$, we get that $\mathcal R[v_\lambda,\lambda^2\,\phi_\lambda]=\mathcal R[v,\phi]$ for any $\lambda>0$ and hence
\[
\L_{-\,\gamma,d}^1=\(\mathsf K^*_{q,d}\)^{-\gamma}\,,
\]
which completes the proof of \eqref{Eq:KLT-GNS2}.

%%%%%%%%%%%%%%%%%%%%%%%%%%%%%%%%%%%%%%%%%%%%%%%%%%%%%%%%%%%%%%%%%%%%%%
%%%%%%%%%%%%%%%%%%%%%%%%%%%%%%%%%%%%%%%%%%%%%%%%%%%%%%%%%%%%%%%%%%%%%%
\bigskip\noindent{\bf Acknowledgements.} J.D.~and M.J.E.~have been partially supported by ANR grants \emph{CBDif} and \emph{NoNAP}. They thank the Mittag-Leffler Institute, where part of this research was carried out, for hospitality.\\[6pt]
{\sl\small \copyright~2013 by the authors. This paper may be reproduced, in its entirety, for non-commercial purposes.}
%%%%%%%%%%%%%%%%%%%%%%%%%%%%%%%%%%%%%%%%%%%%%%%%%%%%%%%%%%%%%%%%%%%%%%
%%%%%%%%%%%%%%%%%%%%%%%%%%%%%%%%%%%%%%%%%%%%%%%%%%%%%%%%%%%%%%%%%%%%%%
%\nocite*
%\bibliographystyle{siam}\small
%\bibliography{References}

%%%%%%%%%%%%%%%%%%%%%%%%%%%%%%%%%%%%%%%%%%%%%%%%%%%%%%%%%%%%%%%%%%%%%%
%%%%%%%%%%%%%%%%%%%%%%%%%%%%%%%%%%%%%%%%%%%%%%%%%%%%%%%%%%%%%%%%%%%%%%
\end{document}